\documentclass[11pt,a4paper,twoside]{article}
\usepackage[hmarginratio=1:1,bmargin=1.3in]{geometry}
\usepackage[OT2,T1]{fontenc}
\usepackage{amsmath, amsthm}
\usepackage{amsfonts}
\usepackage{amssymb}
\usepackage[all]{xy}
\usepackage{graphicx}
\usepackage{xcolor}
\usepackage[nottoc, notlof, notlot]{tocbibind}
\usepackage{array}
\usepackage{url}
\usepackage{rotating}
\usepackage{fancyhdr}
\usepackage{fancyhdr}
\usepackage{titlesec}
\usepackage{xfrac} 
\usepackage{multicol}
\usepackage{comment}
\usepackage{cases}
\usepackage{hyperref}

\title{Transfer principles for Galois cohomology and Serre's conjecture II}
\author{Diego Izquierdo\\
\small
Ecole polytechnique\footnote{At Universit\'e Paris Cit\'e since September 2025.} \\
\small
\texttt{izquierdo@imj-prg.fr}
\and
 Giancarlo Lucchini Arteche\\
 \small
 Universidad de Chile\\
  \small
\texttt{luco@uchile.cl} 
 }
\date{}

\titleformat{\section}[hang]{\center\Large\bf}{\thesection.}{0.5cm}{}

\DeclareSymbolFont{cyrletters}{OT2}{wncyr}{m}{n}
\DeclareMathSymbol{\Sha}{\mathalpha}{cyrletters}{"58}
\DeclareMathSymbol{\Brusse}{\mathalpha}{cyrletters}{"42}

\theoremstyle{plain}
\newtheorem{theorem}{Theorem}[section]
\newtheorem{MT}{Theorem}
\newtheorem{lemma}[theorem]{Lemma}
\newtheorem{proposition}[theorem]{Proposition}
\newtheorem{corollary}[theorem]{Corollary}
\newtheorem{definition}[theorem]{Definition}

\newtheorem{conj}[theorem]{Conjecture}

\newtheorem*{cor}{Corollary}

\theoremstyle{definition}
\newtheorem{remarque}[theorem]{Remark}

\newcommand \gal {{\rm{Gal\,}}}

\renewcommand{\cal}[1]{\mathcal{#1}}

\newcommand \Br {{\rm{Br}}}
\renewcommand \dim {{\rm{dim\,}}}
\newcommand \cd {{\rm{cd\,}}}

\newcommand{\SL}{\mathrm{SL}}

\newcommand{\K}{\mathrm{K}^\mathrm{M}}
\newcommand{\kk}{\mathrm{k}^\mathrm{M}}

\newcommand{\HS}{\mathrm{HS}}
\newcommand{\PHS}{\mathrm{PHS}}

\newcommand{\ad}{\mathrm{ad}}
\newcommand{\sep}{\mathrm{sep}}
\newcommand{\Csc}{C_\mathrm{sc}}
\newcommand{\bb}[1]{\mathbb{#1}}

\newcommand \Z {{\mathbb Z}}

\newcommand \N {{\mathbb N}}

\renewcommand{\ker}{\mathrm{Ker}}

\usepackage{marvosym}

\newcommand{\tK}{\tilde K}
\newcommand{\tL}{\tilde L}
\newcommand{\cO}{\mathcal{O}}
\newcommand{\kM}{\mathrm{k}^{\mathrm{M}}}
\newcommand{\tas}{\tilde\alpha_1,\dots,\tilde\alpha_\delta}
\newcommand{\mono}{\tilde\lambda_{j_1,\dots,j_\delta}^p\,\tilde\alpha_1^{j_1}\cdots\tilde\alpha_\delta^{j_\delta}}

\begin{document}

\maketitle

\begin{abstract}
In this article, we prove several transfer principles for the cohomological dimension of fields. Given a fixed field $K$ with finite cohomological dimension $\delta$, the two main ones allow to:
\begin{itemize}
    \item[-] construct totally ramified extensions of $K$ with cohomological dimension $\leq \delta - 1$ when $K$ is a complete discrete valuation field;
    \item[-] construct algebraic extensions of $K$ with cohomological dimension $\leq \delta-1$ and satisfying a norm condition.
\end{itemize}
We then apply these results to Serre's conjecture II and to some variants for fields of any cohomological dimension that are inspired by conjectures of Kato and Kuzumaki. In particular, we prove that Serre's conjecture II for characteristic $0$ fields implies Serre's conjecture II for positive characteristic fields.
\\

\textbf{MSC Classes:} 11E72, 12G05, 12G10, 14G05, 19D45, 20G10.\\

\end{abstract}

\section{Introduction}

\subsection*{Kato and Kuzumaki's conjectures}

In 1986, in the article \cite{KK}, Kato and Kuzumaki stated a set of conjectures which aimed at giving a diophantine characterization of cohomological dimension of fields. For this purpose, they introduced some properties of fields which are variants of the classical $C_i$-property and which involve Milnor $\mathrm{K}$-theory and projective hypersurfaces of small degree. They hoped that those properties would characterize fields of small cohomological dimension.\\

More precisely, fix a field $K$ and two non-negative integers $q$ and $i$. Let $\K_q(K)$ be the $q$-th Milnor $\mathrm{K}$-group of $K$. For each finite extension $L$ of $K$, one can define a norm morphism $N_{L/K}: \K_q(L)\rightarrow \K_q(K)$ (see Section 1.7 of \cite{Kat}). Thus, if $Z$ is a scheme of finite type over $K$, one can introduce the subgroup $N_q(Z/K)$ of $\K_q(K)$ generated by the images of the norm morphisms $N_{L/K}$ when $L$ runs through the finite extensions of $K$ such that $Z(L) \neq \emptyset$. One then says that the field $K$ is $C_i^q$ if, for each $n \geq 1$, for each finite extension $L$ of $K$ and for each hypersurface $Z$ in $\mathbb{P}^n_{L}$ of degree $d$ with $d^i \leq n$, one has $N_q(Z/L) = \K_q(L)$. For example, the field $K$ is $C_i^0$ if, for each finite extension $L$ of $K$, every hypersurface $Z$ in $\mathbb{P}^n_{L}$ of degree $d$ with $d^i \leq n$ has a zero-cycle of degree 1. The field $K$ is $C_0^q$ if, for each tower of finite extensions $M/L/K$, the norm morphism $N_{M/L}:\K_q(M)\rightarrow \K_q(L)$ is surjective.\\

Kato and Kuzumaki conjectured that, for $i \geq 0$ and $q\geq 0$, a perfect field is $C_i^q$ if, and only if, it is of cohomological dimension at most $i+q$. This conjecture generalizes a question raised by Serre in \cite{SerreCohGal} asking whether the cohomological dimension of a $C_i$-field is at most $i$. As it was already pointed out at the end of Kato and Kuzumaki's original paper \cite{KK}, Kato and Kuzumaki's conjecture for $i=0$ follows from the Bloch-Kato conjecture (which has been established by Rost and Voevodsky, cf.~\cite{Riou}): in other words, a perfect field is $C_0^q$ if, and only if, it is of cohomological dimension at most $q$. However, it turns out that the conjectures of Kato and Kuzumaki are wrong in general. For example, Merkurjev constructed in \cite{Merkurjev} a field of characteristic 0 and of cohomological dimension 2 which did not satisfy property $C^0_2$. Similarly, Colliot-Th\'el\`{e}ne and Madore produced in \cite{CTM} a field of characteristic 0 and of cohomological dimension 1 which did not satisfy property $C^0_1$. These counter-examples were all constructed by a method using transfinite induction due to Merkurjev and Suslin. 

\subsection*{Higher versions of Serre's conjectures}

In the article \cite{ILA}, we introduced several variants of Kato and Kuzumaki's $C_i^q$ properties, by replacing hypersurfaces of low degree by homogeneous spaces of linear algebraic groups. One of those variants can be stated as follows: given a non-negative integer $q$, we say that a field $K$ has the $C_{\PHS}^q$ property if, for each finite extension $L$ of $K$ and for each principal homogeneous space $Z$ under a smooth linear connected algebraic group over $L$, one has $N_q(Z/L) = K_q^M(L)$. According to the Main Theorem of \cite{ILA}, the $C_{\PHS}^q$ property characterizes perfect fields with cohomological $\leq q+1$ and is hence a good replacement for Kato and Kuzumaki's $C_1^q$ property.

In the particular case where $q=0$, our result recovers the zero-cycle version of Serre's conjecture I, which is a classical theorem by Steinberg \cite{Steinberg}:

\begin{theorem}[Serre's conjecture I]\label{conj Serre I}
Let $K$ be a perfect field of cohomological dimension $\leq 1$. Then every principal homogeneous space under a connected linear $K$-group has a rational point.
\end{theorem}

The Main Theorem of \cite{ILA} is therefore in some sense a version of Serre's conjecture I for higher-dimensional fields. It is then natural to ask whether one can find good replacements for the $C_2^q$ property that would characterize fields with cohomological dimension $\leq q+2$, while recovering a version for higher-dimensional fields of the so-called Serre's conjecture II:

\begin{conj}[Serre's conjecture II]\label{conj Serre II}
Let $K$ be a field of cohomological dimension $\leq 2$. Then every principal homogeneous space under a semisimple simply connected $K$-group has a rational point.
\end{conj}

Contrary to Serre's conjecture I, this one remains open, although a lot has been done in particular cases, whether specifying the type of the group or the field of definition (cf.~the lecture notes \cite{GilleLNM} for a survey on the topic). For instance, Bayer-Fluckiger and Parimala \cite{BP} proved it for classical groups (types $A,B,C$ and $D$ with no triality) and groups of type $F_4$ and $G_2$ over perfect fields, while the same result was obtained for imperfect fields by Berhuy, Frings and Tignol \cite{BFT}. Quasi-split groups of exceptional types other than $E_8$ have been treated by Garibaldi \cite{garibaldi}, Knus, Merkurjev, Rost and Tignol \cite{Invol} (cf.~also \cite{garibaldinote}), Chernousov \cite{ChernousovSerreII} and Gille \cite{GilleSerreIIqd} (cf.~also \cite[Thm.~6.0.1]{GilleLNM}). On the other hand, classical work by Kneser \cite{KneserH1LocalI, KneserH1LocalII}, Bruhat and Tits \cite{BTIII}, Harder \cite{HarderHasseCar0I, HarderHasseCar0II}, and Chernousov \cite{ChernousovE8} yields the conjecture for $p$-adic fields and totally imaginary number fields, while the work of Colliot-Th\'el\`ene, Gille and Parimala \cite{ColliotGilleParimala}, and of de Jong, He and Starr \cite{dJHS} implies the conjecture for diverse fields of geometric nature, including function fields of complex algebraic surfaces.\\

In analogy with our generalization of Serre's conjecture I in \cite{ILA}, we are led to the following conjecture:

\begin{conj}[Higher Serre's conjecture II]\label{conj higher Serre II}
Let $K$ be a field of cohomological dimension $\leq q+2$. Then for every principal homogeneous space $Z$ under a semisimple simply connected $K$-group we have $N_q(Z/K)=\K_q(K)$.
\end{conj}

Note that a converse to this statement will be proved to hold for perfect fields later in the article (Proposition \ref{1An}).\\

As before, the case where $q=0$ corresponds to a weakening of the original conjecture, where rational points have been replaced by zero-cycles of degree $1$. However, the upshot of the higher version of Serre's conjecture II is a gain in flexibility. Indeed, it allows to work with fields of \emph{arbitrary} cohomological dimension, and hence it allows to argue by focusing on the \emph{fields} involved instead of \emph{groups}. This is in full contrast with the usual approaches to this conjecture and with the methods used in \cite{ILA}.

For instance, given a field $K$ of positive characteristic and cohomological dimension $\leq 2$, one can always construct a complete discrete valuation field $\tilde K$ of characteristic $0$ with residue field $K$. One would then like to be able to deduce Serre's conjecture II for $K$ (whether with rational points or zero-cycles of degree $1$) from a statement for $\tilde K$. But this is a field of cohomological dimension $\leq 3$, so one cannot apply the classical conjecture to $\tilde K$. The higher version of the conjecture provides a method to avoid this difficulty.

\subsection*{Transfer principles}
In order to exploit the above-mentioned flexibility and to prove some instances of the Higher Serre's conjecture II, we need to construct, given a field of fixed cohomological dimension, some associated fields with prescribed cohomological dimension and suitable additional properties. One of the main objectives of this article is to present the following ``transfer principles'', which go in this direction, but that are of independent interest.

\begin{MT}[From positive to zero characteristic, Theorem \ref{thm car p vers car 0}]~\label{MTB}\\
Let $\tilde K$ be a complete discrete valuation field of characteristic $0$ with countable residue field $K$ of cohomological dimension $\delta$. Then there exists a totally ramified extension $\tilde K_\dagger/\tilde K$ with cohomological dimension $\delta$.
\end{MT}

\noindent Recall that a field is called $\ell$-special for some prime number $\ell$ if its absolute Galois group is a pro-$\ell$-group.

\begin{MT}[From higher to lower cohomological dimension, Theorem \ref{th countable}]~\label{MTC}\\
Let $\delta\geq 1$ be an integer, $\ell$ a prime number and $K$ an $\ell$-special countable field of characteristic $0$ and with cohomological dimension $\delta$. For each $x \in K^\times$, there exists an algebraic extension $K_x$ of $K$ that has cohomological dimension $\leq (\delta-1)$ and such that $x \in N_{L/K}(L^\times)$ for every finite subextension $L$ of $K_x/K$.
\end{MT}

Even if both statements consider a \emph{countable} field $K$, this should not be seen as a constraint. Indeed, classical results from Model Theory---in particular the Löwenheim-Skolem Theorem---allow one to easily reduce from general perfect fields to countable fields. In Section \ref{sec unc}, we give a field-theoretic proof of such a reduction result (Proposition \ref{prop unc}) that also applies to imperfect fields and whose arguments will be useful in the proofs of Theorems \ref{MTB} and \ref{MTC}.\\

Theorem \ref{MTC} yields a ``transfer principle'' for the norm groups $ N_q(Z/K)$, which has some natural consequences on Kato and Kuzumaki's conjectures.

\begin{MT}[Transfer principle for norm groups, Theorem \ref{th transfer principle car 0}]\label{MTD}~\\
Let $m,n,q \geq 0$ be three integers with $n \geq m \geq 1$ and $K$ a field of characteristic 0 and with cohomological dimension $\leq n$. Let $Z$ be a $K$-variety and assume that there exists a countable subfield $K_0$ of $K$ and a form $Z_0$ of $Z$ over $K_0$ satisfying the following condition:
\begin{itemize}
\item[$(\star)$] For every countable subextension $L$ of $K/K_0$, and for every algebraic extension $M/L$ of cohomological dimension $\leq m$, we have $N_{q}(Z_{0,M}/M)=\K_{q}(M)$.
\end{itemize}
Then $N_{n-m+q}(Z/K)=\K_{n-m+q}(K)$.
\end{MT}

A specific application is the following statement:

\begin{cor}[Corollary \ref{cor KK fgQ}]
Let $K$ be a field with finite transcendence degree over $\mathbb{Q}$. Assume that, for every $i\geq 1$ and for every algebraic extension $L/K$ of cohomological dimension $i$, every hypersurface $Z$ in $\mathbb{P}^n_{L}$ of degree $d$ with $d^i\leq n$ has a zero-cycle of degree 1. Then Kato and Kuzumaki's conjectures hold for $K$.
\end{cor}

Theorem \ref{MTD} also allows to immediately recover, \emph{using completely different techniques}, the theorem of \cite{ILA} settling the $C^q_\mathrm{PHS}$ property for perfect fields of cohomological dimension $\leq q+1$ (Corollary \ref{cor Serre I}). But its most beautiful and powerful applications concern Serre's conjecture II and its higher version \ref{conj higher Serre II}.

\subsection*{Applications to Serre's conjecture II and its higher version}

The main application of Theorem \ref{MTD} is the following conditional result concerning Conjecture \ref{conj higher Serre II}:

\begin{MT}[Theorem \ref{thm conditionnel}]\label{MTE}
Assume that the classical Serre's conjecture II holds for countable fields of characteristic $0$. Then for every field $K$ of cohomological dimension $\leq q+2$ and for every principal homogeneous space $Z$ under a semisimple simply connected $K$-group we have $N_q(Z/K)=\K_q(K)$.
\end{MT}

This result describes the behaviour of principal homogeneous spaces under semisimple simply connected groups over fields with bounded cohomological dimension. It turns out that it also holds when one restricts it (together with the assumption on Serre's conjecture II) to semisimple simply connected isotypical groups of a given type. Hence, using the currently known cases of Conjecture \ref{conj Serre II}, we get the following \emph{unconditional} corollary:

\begin{cor}[Corollary \ref{cor incond}]
    Let $K$ be a field of cohomological dimension $\leq q+2$. Then for every principal homogeneous space $Z$ under a semisimple simply connected $K$-group without factors of types $E_6$ or $E_7$, we have $N_q(Z/K)=\K_q(K)$.
\end{cor}

The last main result of this article is an application of the transfer principle given by Theorem \ref{MTB} to the classical Serre's conjecture II:

\begin{MT}[Theorem \ref{thm Serre II clasico}]\label{MTF}
 If Serre's conjecture II holds for countable fields of characteristic 0, then it holds for arbitrary fields.
\end{MT}

Again, this result also holds when one restricts to principal homogeneous spaces under semisimple simply connected isotypical groups of a given type. It therefore allows to \emph{unconditionally} recover Serre's conjecture II for groups without factors of types $E_6$, $E_7$, $E_8$ or trialitarian $D_4$ over imperfect fields from the analoguous result over perfect fields, which was settled by Bayer-Fluckiger and Parimala in \cite{BP}. This consequence of Theorem \ref{MTF} was already proved in \cite{BFT} by a completely different method based on a case by case study following the classification of semisimple simply connected groups instead of focusing on fields.

\subsection*{Organization of the article}

In section \ref{sec 2}, we introduce a certain number of preliminaries and notations regarding the cohomological dimension of fields, Milnor $\mathrm{K}$-theory and the norm groups $N_q(Z/K)$. Section \ref{sec 3} is dedicated to the proofs of the transfer principles for Galois cohomology given by Theorems \ref{MTB} and \ref{MTC}. In Section \ref{sec 4}, we deduce the transfer principle for norm groups given by Theorem \ref{MTD}, and we explain how to get the applications to Serre's conjecture II and its higher variants given by Theorems \ref{MTE} and \ref{MTF}.

\subsection*{Acknowledgements}

We thank Skip Garibaldi, Philippe Gille, Olivier Wittenberg and an anonymous referee for their comments and support. The second author's research was partially supported by ANID via FONDECYT Grants 1210010 and 1240001.

\section{Preliminaries and notations}\label{sec 2}

\subsection{The cohomological dimension of a field}
The cohomological dimension  $\mathrm{cd}(K)$ of a perfect field $K$ is the cohomological dimension of its absolute Galois group. In other words, it is the smallest integer $\delta$ (or $\infty$ if such an integer does not exist) such that $H^n(K,M)=0$ for every $n>\delta$ and each finite Galois module $M$. 

It is more complicated to define a good notion for the cohomological dimension of an imperfect field. This was first done by Kato in \cite{Kato}.

\begin{definition}
Let $K$ be any field. 
\begin{itemize}
\item[(i)] Let $\ell$ be a prime number different from the characteristic of $K$. The \emph{$\ell$-cohomological dimension} $\mathrm{cd}_{\ell}(K)$ and the \emph{separable $\ell$-cohomological dimension} $\mathrm{sd}_{\ell}(K)$ of $K$ are both the $\ell$-cohomological dimension of the absolute Galois group of $K$.
\item[(ii)] (Kato, \cite{Kato}; Gille, \cite{gille}). Assume that $K$ has characteristic $p>0$. Let $\Omega^i_K$ be the $i$-th exterior product over $K$ of the absolute differential module $\Omega^1_{K/\mathbb{Z}}$ and let $H^{i+1}_p(K)$ be the cokernel of the morphism $\mathfrak{p}^i_K\,:\, \Omega^i_K \to \Omega^i_K/d(\Omega^{i-1}_K)$ defined by
\[x \frac{dy_1}{y_1} \wedge ... \wedge \frac{dy_i}{y_i} \mapsto (x^p-x) \frac{dy_1}{y_1} \wedge ... \wedge \frac{dy_i}{y_i} \mod d(\Omega^{i-1}_K),\]
for $x\in K$ and $y_1,...,y_i \in K^{\times}$. The \emph{$p$-cohomological dimension} $\mathrm{cd}_{p}(K)$ of $K$ is the smallest integer $\delta$ (or $\infty$ if such an integer does not exist) such that $[K:K^p]\leq p^\delta$ and $H^{\delta+1}_p(L)=0$ for all finite extensions $L$ of $K$. The \emph{separable $p$-cohomological dimension} $\mathrm{sd}_{p}(K)$ of $K$ is the smallest integer $\delta$ (or $\infty$ if such an integer does not exist) such that $H^{\delta+1}_p(L)=0$ for all finite separable extensions $L$ of $K$.
\item[(iii)] The \emph{cohomological dimension} $\mathrm{cd}(K)$ of $K$ is the supremum of all the $\mathrm{cd}_{\ell}(K)$'s when $\ell$ runs through all prime numbers. The \emph{separable cohomological dimension} $\mathrm{sd}(K)$ of $K$ is the supremum of all the $\mathrm{sd}_{\ell}(K)$'s when $\ell$ runs through all prime numbers. 
\end{itemize}
\end{definition}

For further use, we recall the following classical criterion for cohomological dimension, which follows from \cite[I.3.3, Cor.~1, I.4.1, Prop.~21]{SerreCohGal}).

\begin{proposition}\label{prop Serre}
Let $K$ be any field and let $\ell$ be a prime number different from the characteristic of $K$. Then $K$ has $\ell$-cohomological dimension $\leq \delta$ if and only if, for every finite separable extension $L/K$ with $[L:K]$ coprime to $\ell$ and containing a primitive $\ell$-th root of unity, $H^{\delta+1}(L,\mu_\ell^{\otimes(\delta+1)})=0$.
\end{proposition}

\subsection{Milnor $\mathrm{K}$-theory}\label{sec K-th}
Let $K$ be any field and let $q$ be a non-negative integer. The $q$-th Milnor $\mathrm{K}$-group of $K$ is by definition the group $\K_0(K)=\mathbb{Z}$ if $q =0$ and:
$$\K_q(K):= \underbrace{K^{\times} \otimes_{\mathbb{Z}} ... \otimes_{\mathbb{Z}} K^{\times}}_{q \text{ times}} / \left\langle x_1 \otimes ... \otimes x_q | \exists i,j, i\neq j, x_i+x_j=1 \right\rangle$$
if $q>0$. For $x_1,...,x_q \in K^{\times}$, the symbol $\{x_1,...,x_q\}$ denotes the class of $x_1 \otimes ... \otimes x_q$ in $\K_q(K)$. More generally, for $r$ and $s$ non-negative integers such that $r+s=q$, there is a natural pairing:
$$\K_r(K) \times \K_s(K) \rightarrow \K_q(K)$$
which we will denote $\{\cdot, \cdot\}$.\\

When $L$ is a finite extension of $K$, one can construct a norm homomorphism
\[N_{L/K}: \K_q(L) \rightarrow \K_q(K),\]
satisfying the following properties (cf.~\cite[\S 1.7]{Kat} or \cite[\S 7.3]{GS}):
\begin{itemize}
\item[$\bullet$] For $q=0$, the map $N_{L/K}: \K_0(L) \rightarrow \K_0(K)$ is given by multiplication by $[L:K]$.
\item[$\bullet$] For $q=1$, the map $N_{L/K}: \K_1(L) \rightarrow \K_1(K)$ coincides with the usual norm $L^{\times} \rightarrow K^{\times}$.
\item[$\bullet$] If $r$ and $s$ are non-negative integers such that $r+s=q$, we have $N_{L/K}(\{x,y\})=\{x,N_{L/K}(y)\}$ for $x \in \K_r(K)$ and $y\in \K_s(L)$.
\item[$\bullet$] If $M$ is a finite extension of $L$, we have $N_{M/K} = N_{L/K} \circ N_{M/L}$.
\end{itemize}
Recall also that Milnor $\mathrm{K}$-theory is endowed with residue maps (cf.~\cite[\S 7.1]{GS}). Indeed, when $\tilde K$ is a complete discrete valuation field with ring of integers $\cal O_{\tilde K}$ and residue field $K$, there exists a unique residue morphism:
$$\partial: \K_{q}(\tilde K) \rightarrow \K_{q-1}(K)$$
such that, for each uniformizer $\pi$ and for all units $\tilde u_2,...,\tilde u_{q}\in \cal O_{\tilde K}^{\times}$ whose images in $K$ are denoted $u_2,\ldots,u_{q}$, one has:
$$\partial (\{\pi,\tilde u_2,\ldots,\tilde u_{q}\})=\{ u_2,\ldots,u_{q}\}.$$
The kernel of $\partial$ is the subgroup $U_{q}(\tilde K)$ of $\K_{q}(\tilde K)$ generated by symbols of the form $\{x_1,\ldots,x_{q}\}$ with $x_1,\ldots,x_{q} \in \cal O_{\tilde K}^\times$. We denote by $U_{q}^i(\tilde K)$ the subgroup of $\K_{q}(\tilde K)$ generated by those symbols that lie in $U_q(\tilde K)$ and that are of the form $\{x_1,\ldots,x_q\}$ with $x_1 \in 1+\pi^i\cal O_{\tilde K}$ and $x_2,\ldots,x_q \in \tilde K^{\times}$.

Let us finally recall the description given by Kato of the $\mathrm{K}$-theory modulo $p$ of $\tilde K$ when $K$ is of characteristic $p>0$ and $\tilde K$ contains a primitive $p$-th root of unity $\zeta_p$. For that purpose, we set $\kk_q(\tilde K):=\K_q(\tilde K)/p$ and $u_q^i(\tilde K):=U_q^i(\tilde K)/p$. We fix a uniformizer $\pi\in\tilde K$, and for each $a\in K$, we denote by $\tilde a$ a lift of $a$ to $\cal O_{\tilde K}$. Finally, we set $\epsilon_{\tilde K}:=\frac{v_{\tilde K}(p)p}{p-1}$, where $v_{\tilde K}$ denotes the valuation on $\tilde K$. Then, for every integer $q\geq 0$, we have the following isomorphisms (cf.~\cite[p.~218]{Kato}): 
\begin{align}
\label{eq Kato i}
\rho^{q+1}_0 : \kk_{q+1}(K)\oplus \kk_q(K) &\to \kk_{q+1}(\tilde K)/ u^1_{q+1}(\tilde K), \\
\nonumber\left(\{y_1,y_2,\dots,y_{q+1}\},0\right) &\mapsto \{\tilde y_1,\tilde y_2,\dots,\tilde y_{q+1} \}, \\
\nonumber\left(0,\{y_1,y_2,\dots,y_q\}\right) &\mapsto \{\pi,\tilde y_1,\tilde y_2,\dots,\tilde y_q \}; \\
\intertext{for every $0<i<\epsilon_{\tilde K}$ such that $(i,p)=1$,}
\label{eq Kato ii}
\rho^{q+1}_i: \Omega^{q}_K &\to u^i_{q+1}(\tilde K)/ u^{i+1}_{q+1}(\tilde K), \\
\nonumber x \,\frac{dy_1}{y_1}\wedge \frac{dy_2}{y_2} \wedge \dots \wedge \frac{dy_q}{y_q}  &\mapsto \{1+\pi^i\tilde x,\tilde y_1,\tilde y_2, \dots, \tilde y_q \}; \\
\intertext{for every $0<i<\epsilon_{\tilde K}$ such that $p|i$, }
\label{eq Kato iii}
\rho^{q+1}_i : \Omega^{q}_K/\ker(d) \oplus \Omega^{q-1}_K/\ker(d) &\to  u^i_{q+1}(\tilde K)/ u^{i+1}_{q+1}(\tilde K), \\
\nonumber\left( x \,\frac{dy_1}{y_1}\wedge \frac{dy_2}{y_2} \wedge \dots \wedge \frac{dy_q}{y_q}, 0 \right) &\mapsto  \{1+\pi^i\tilde x,\tilde y_1,\tilde y_2, \dots, \tilde y_q \}, \\
\nonumber\left( 0, x \,\frac{dy_1}{y_1}\wedge \frac{dy_2}{y_2} \wedge \dots \wedge \frac{dy_{q-1}}{y_{q-1}} \right) &\mapsto  \{\pi,1+\pi^i\tilde x,\tilde y_1,\tilde y_2, \dots, \tilde y_{q-1} \}; \\
\intertext{for $i=\epsilon_{\tilde K}$,}
\label{eq Kato iv}
\rho^{q+1}_{i} : H_p^{q+1}(K)\oplus H_p^{q}(K) &\to u^{i}_{q+1}(\tilde K), \\
\nonumber\left( x \,\frac{dy_1}{y_1}\wedge \frac{dy_2}{y_2} \wedge \dots \wedge \frac{dy_q}{y_q}, 0 \right) &\mapsto  \{1+(\zeta_p-1)^p\tilde x,\tilde y_1,\tilde y_2, \dots, \tilde y_q \}, \\
\nonumber\left( 0, x \,\frac{dy_1}{y_1}\wedge \frac{dy_2}{y_2} \wedge \dots \wedge \frac{dy_{q-1}}{y_{q-1}} \right) &\mapsto  \{\pi,1+(\zeta_p-1)^p\tilde x,\tilde y_1,\tilde y_2, \dots, \tilde y_{q-1} \}.
\end{align}

\subsection{The norm groups $N_q(Z/K)$}
We conclude these preliminaries by studying the following norm groups, which were originally defined by Kato and Kuzumaki in \cite{KK}:

\begin{definition}
Let $K$ be a field and let $q$ be a non-negative integer. For each $K$-scheme $Z$ of finite type, we define the \emph{norm group} $N_q(Z/K)$ as the subgroup of $\K_q(K)$ generated by the images of the maps $N_{L/K}: \K_q(L) \rightarrow \K_q(K)$ when $L$ runs through the finite extensions of $K$ such that $Z(L)\neq \emptyset$.
\end{definition}

In this article, we will use the following elementary lemmas about the norm groups.

\begin{lemma}\label{lem produit}
Let $K$ be a field and let $Y$ and $Z$ be two $K$-varieties. Assume that $N_q(Y_{L}/L)=N_q(Z_{L}/L)=\K_q(L)$ for every finite extension $L$ of $K$. Then $N_q(Y\times_K Z/K)=\K_q(K)$.
\end{lemma}

\begin{proof}
Fix $z\in \K_q(K)$. Since $N_q(Y/K)=\K_q(K)$, there exist finite extensions $L_1,\dots,L_r$ of $K$ and elements $z_1\in\K_q(L_1),\dots,z_r\in\K_q(L_r)$ such that:
\[\begin{cases} 
\forall\,i, \; Y(L_i)\neq\emptyset, \\
z = \sum_i N_{L_i/K}(z_i).
\end{cases}\]
Now, for each $i$, we have $N_q(Z_{L_i}/L_i)=\K_q(L_i)$, and hence there exist finite extensions $L_{i,1},\dots,L_{i,r_i}$ of $L_i$ and elements $z_{i,1}\in\K_q(L_{i,1}), \dots z_{i,r_i}\in\K_q(L_{i,r_i})$ such that:
\[\begin{cases} 
\forall\,j, \; Z(L_{i,j})\neq\emptyset, \\
z_i=\sum_j N_{L_{i,j}/L_{i}}(z_{i,j}).
\end{cases} \]
Hence $z=\sum_{i,j} N_{L_{i,j}/K}(z_{i,j})$, and since both $Y$ and $Z$ have points in all the $L_{i,j}$'s, we get $z\in N_q(Y\times_K Z/K)$.
\end{proof}

\begin{lemma} \label{lem p Sylow}
Let $K$ be a field, $q\geq 0$ an integer and $Z$ a $K$-variety. For every prime number $p$, fix an extension $K_p/K$ corresponding to a $p$-Sylow subgroup of $\gal(K^\sep/K)$. Assume that $N_q(Z_{K_p}/K_p)=\K_q(K_p)$ for every prime $p$. Then $N_q(Z/K)=\K_q(K)$.
\end{lemma}

\begin{proof}
Take $z \in \K_q(K)$ and fix a prime number $p$. Since $N_q(Z_{K_p}/K_p)=\K_q(K_p)$, we can find finite extensions $L_1,\ldots, L_r$ of $K_p$ such that
\[\begin{cases}
\forall\,i,\; Z(L_i)\neq \emptyset,\\
z \in \langle N_{L_i/K_p}(\K_q(L_i)) : 1\leq i\leq r\rangle.
\end{cases}\]
We can then also find a finite subextension $K'_p$ of $K_p/K$ together with finite extensions $L'_1,\ldots, L'_r$ of $K'_p$, with $L'_i\subset L_i$, such that $z \in \langle N_{L'_i/K'_p}(\K_q(L'_i)) : 1\leq i\leq r\rangle$ and $Z(L'_i)\neq \emptyset$ for every $1\leq i
\leq r$. In particular:
$$[K'_p:K]\cdot z\in \langle N_{L'_i/K}(\K_q(L'_i)) : 1\leq i\leq r\rangle \subset N_q(Z/K).$$
Since $[K'_p:K]$ is not divisible by $p$ and since the previous argument can be done for every prime number $p$, we deduce that $z\in N_q(Z/K)$.
\end{proof}

\section{Transfer principles for cohomological dimension}\label{sec 3}

In this section, we prove the transfer principles for Galois cohomology given by Theorems \ref{MTB} and \ref{MTC}. These results allow to move from positive characteristic to characteristic $0$ fields, and from fields with higher cohomological dimension to lower cohomological dimension.

As noted in the Introduction, these two results concern a \emph{countable} field $K$, so that we will also need a statement that allows us to move from uncountable to countable fields. This is the subject of Section \ref{sec unc}.

\subsection{From uncountable to countable fields}\label{sec unc}

In the case of perfect fields, the following result seems to be well-known to experts in Model Theory and a proof can be given by using the L\"owenheim-Skolem Theorem and by slightly adapting the proof of \cite[Thm.~2.5]{Zoe}\footnote{We thank an anonymous referee for pointing out this reference.}. We give here below an elementary algebraic proof that also applies to imperfect fields. Similar arguments will also be used in the proofs of Theorems \ref{MTB} and \ref{MTC}.

\begin{proposition}\label{prop unc}
Let $K$ be a field with cohomological dimension $\delta$. Let $K_0$ be a countable subfield of $K$. Then there exists a countable subextension $K_\infty$ of $K/K_0$ that has cohomological dimension $\leq \delta$.
\end{proposition}

\begin{remarque}
By adapting the subsequent proof of Proposition \ref{prop unc} and by using transfinite induction, it is possible to replace $K_0$ by any infinite subfield of $K$ in the statement and conclude that there exists a subextension $K_\infty$ of $K/K_0$ that has the same cardinality as $K_0$ and that has cohomological dimension $\leq \delta$.
\end{remarque}

In order to deal with the case of imperfect fields, we need first the following result.

\begin{lemma}\label{lem p-base}
Let $K$ be a field of characteristic $p>0$ such that $[K:K^p]=p^\delta$ and let $\alpha_1,\ldots,\alpha_\delta$ be a $p$-basis. Let $K_0$ be a countable subfield of $K$. Then there exists a countable subextension $\Theta(K_0)$ of $K/K_0$ with the same $p$-basis as $K$.
\end{lemma}

\begin{proof}
For each $x\in K$, write
\[x=\sum_{0\leq i_1,\ldots,i_\delta<p}\lambda^p_{x,i_1,\ldots,i_\delta}\alpha_1^{i_1}\cdots\alpha_\delta^{i_\delta},\]
with $\lambda_{x,i_1,\ldots,i_\delta}\in K$. We then define, for an arbitrary subfield $L$ of $K$ containing $\alpha_1,\dots,\alpha_\delta$,
\[\theta(L):=L(\lambda_{x,i_1,\ldots,i_\delta}:0\leq i_1,\ldots,i_\delta<p, x\in L).\]
Consider now the field $K_1:=K_0(\alpha_1,\ldots,\alpha_\delta)$ and, for $i\geq 1$, define $K_{i+1}:=\theta(K_i)$. We claim that
\[\Theta(K_0):=\bigcup_{i\in\mathbb{N}} K_i\]
satisfies the statement of the Lemma. Indeed, by construction, it is a countable subextension of $K/K_0$. Moreover, for $x\in \Theta(K_0)$, there exists $i_0\in\N$ such that $x\in K_{i_0}$ and thus, we can always write
\[x=\sum_{0\leq i_1,\ldots,i_\delta<p}\lambda^p_{x,i_1,\ldots,i_\delta}\alpha_1^{i_1}\cdots\alpha_\delta^{i_\delta},\]
with $\lambda_{x,i_1,\ldots,i_\delta}\in K_{i_0+1}\subset \Theta(K_0)$. This proves that $\alpha_1,\dots,\alpha_\delta$ is a $p$-basis of $\Theta(K_0)$.
\end{proof}

\begin{proof}[Proof of Proposition \ref{prop unc}]
Fix an algebraic closure $\overline{K}$ of $K$ and, if $p>0$, a $p$-basis of $K$. All fields considered in this proof, in particular composite fields, will be assumed to be subfields of $\overline{K}$. We also fix the following:
\begin{itemize}
\item  We let $\sigma=(\sigma_1,\sigma_2): \mathbb{N} \rightarrow \mathbb{N}^2$ be a bijection such that $\sigma_1(n)\leq n$ for each $n \in \mathbb{N}$.
\item For each countable subextension $L$ of $K/K_0$, we define $\mathcal{E}_L$ as the set of triples $(M, \ell, a)$ such that $M$ is a finite extension of $L$, $\ell$ is a prime number and $a$ is an element in
\[H^{\delta+1}_\ell(M):=\begin{cases}
H^{\delta+1}(M,\mu_\ell^{\otimes (\delta+1)}) & \text{if } \ell\neq p;\\
H^{\delta+1}_p(M) & \text{if } \ell=p.
\end{cases}\]
Note that such a set is countable. Indeed, the set of finite extensions $M$ of $L$ is countable, the set of prime numbers is also countable, and if $M$ is a finite extension of $L$, it is automatically countable and hence so is the group $H^{\delta+1}_\ell(M)$ (by definition for $\ell=p$ and by the isomorphism $H^{\delta+1}(M,\mu_\ell^{\otimes (\delta+1)})\cong \K_{\delta+1}(M)/\ell$ given by the Bloch-Kato conjecture and the definition of Milnor $K$-theory when $\ell\neq p$). So we may and do assume that the set $\mathcal{E}_L$ comes with a given numbering.
\end{itemize}

Let us now inductively construct an increasing infinite sequence of subfields $(L_i)_{i\geq 1}$ of $K$ in the following way. We set first $L_1:=K_0$. Then, for $i \geq 1$, consider the $\sigma_2(i)$-th term of the set $\mathcal E_{L_{\sigma_1(i)}}$, which we denote by $(M_i,\ell_i,a_i)$. Since $\sigma_1(i)\leq i$, we have $L_{\sigma_1(i)}\subset L_{i}$. Moreover, since $K$ has cohomological dimension $\delta$, we have $a_{i}|_{M_i K}=0$. Hence we may and do define $L'_{i+1}$ as a finitely generated extension of $L_{i}$ contained in $K$ such that $a_{i}|_{M_iL'_{i+1}}=0$, and we set 
\[L_{i+1}:=\begin{cases}
L'_{i+1} & \text{if } p=0;\\
\Theta(L'_{i+1}) & \text{if } p>0.
\end{cases}\]

We introduce the field:
$$K_\infty:= \bigcup_{i\geq 1} L_i.$$
We claim that $K_\infty$ satisfies the conditions of the theorem. Indeed, all the $L_i$'s are countable fields, and hence so is $K_\infty$. Moreover, if $p>0$, we know by Lemma \ref{lem p-base} that $L_i$ has the same $p$-basis as $K$ for $i\geq 2$, and hence $[K_\infty:K_\infty^p]=[K:K^p]\leq p^\delta$. Finally, according to the definition of the cohomological dimension and to Proposition \ref{prop Serre}, in order to prove that $\mathrm{cd}(K_\infty)\leq \delta$, it suffices to check that $H_\ell^{\delta+1}(M_\infty)=0$ for each finite extension $M_\infty/K_\infty$ and each prime number $\ell$. We henceforth fix an element $a \in H_\ell^{\delta+1}(M_\infty)$. We can then find an integer $i \geq 1$, a finite extension $M$ of $L_i$ and an element $b\in H^{\delta+1}_\ell(M)$ such that $MK_\infty=M_\infty$ and $b|_{M_\infty}=a$. Now, the triple $(M,\ell,b)$ is the $j$-th element of $\mathcal E_{L_i}$ for some $j\in\mathbb N$. This implies that $b|_{ML'_{\sigma^{-1}(i,j)+1}}=0$ by construction. By the inclusion $ML'_{\sigma^{-1}(i,j)+1}\subset ML_{\sigma^{-1}(i,j)+1}\subset MK_\infty=M_\infty$, we deduce that $a=b|_{M_\infty}=0$, as wished.
\end{proof}

\subsection{From positive to zero characteristic}\label{sec pos zero}

\begin{theorem}\label{thm car p vers car 0}
Let $\tilde K$ be a complete discrete valuation field of characteristic $0$ with countable residue field $K$ of cohomological dimension $\delta$. Then there exists a totally ramified extension $\tilde K_\dagger/\tilde K$ with cohomological dimension $\delta$.
\end{theorem}

\begin{remarque}
    It is highly likely that the countability assumption of the theorem can be removed by adapting the subsequent proof and by using transfinite induction. Doing so would however significantly increase the technicality of the proof, and for applications we will not need such a general statement.
\end{remarque}

The main ingredient in order to prove Theorem \ref{thm car p vers car 0} is the following proposition.

\begin{proposition}\label{proposition cle}
Let $\tilde K$ be a complete discrete valuation field of characteristic $0$ with infinite residue field $K$ of characteristic $p>0$. Let $\delta$ be the cohomological dimension of $K$ and let $\tilde L/\tilde K$ be a finite unramified Galois extension with residue field extension $L/K$. Assume that $\tilde K$ contains a primitive $p$-th root of unity and that it contains $\sqrt{-1}$ if $p=2$. Then, for any element $a\in\kk_{\delta+1}(\tilde L)=\K_{\delta+1}(\tilde L)/p$, there exists a finite and totally ramified extension $\tilde K'/\tilde K$ of $p$-primary degree such that $a$ is trivial when restricted to $\tilde K'\tilde L$.
\end{proposition}

\subsubsection{Preliminaries to the proof of Proposition \ref{proposition cle}}\label{subsection prel proofs}

In the whole Section \ref{sec pos zero}, we fix an algebraic closure $\overline{\tilde K}$ of $\tilde K$, as well as a primitive $p$-th root of unity $\zeta_p\in\overline{\tilde K}$ and a uniformizer $\pi$ in $\tilde K$. We also recall the notations defined in Section \ref{sec K-th}: for $b\in K$, we denote by $\tilde b$ a lift of $b$ to $\cal O_{\tilde K}$ and we set $\epsilon_{\tilde K}:=\frac{v_{\tilde K}(p)p}{p-1}$.\\

Both in Theorem \ref{thm car p vers car 0} and Proposition \ref{proposition cle}, the field $K$ is assumed to have cohomological dimension $\delta$. In particular we have $\kk_{\delta+1}(K)=0$ and $[K:K^p]= p^{\delta_0}$ with $\delta_0\leq\delta$. We henceforth fix a family $(\alpha_1,\dots, \alpha_\delta)$ in $K$ such that the first $\delta_0$ terms form a $p$-basis of $K/K^p$ and $\alpha_i=0$ for $i>\delta_0$. 

Note that the finite separable extension $L$ of $K$ also satisfies $\kk_{\delta+1}(L)=0$ and has the same $p$-basis as $K$. In particular,
\[\Omega_L^\delta=L\left(\frac{d\alpha_1}{\alpha_1}\wedge \dots \wedge \frac{d\alpha_\delta}{\alpha_\delta}\right)\quad \text{and}\quad \Omega_L^{\delta+1}=0,\]
and hence
\[\Omega_L^\delta/\ker(d)=d(\Omega_L^\delta)=0 \quad \text{and} \quad  H_p^{\delta+1}(L)=0.\]
The isomorphisms $\rho_i:=\rho^{\delta+1}_i$ defined in Section \ref{sec K-th} can therefore be described as follows:
\begin{align}
\label{eq Kato i bis}
\rho_0 : \kk_\delta(L) &\to \kk_{\delta+1}(\tilde L)/ u^1_{\delta+1}(\tilde L), \\
\nonumber\{y_1,y_2,\dots,y_\delta\} &\mapsto \{\pi,\tilde y_1,\tilde y_2,\dots,\tilde y_\delta \}; \\
\intertext{for every $0<i<\epsilon_{\tilde L}$ such that $(i,p)=1$,}
\label{eq Kato ii bis}
\rho_i: \Omega^{\delta}_L &\to u^i_{\delta+1}(\tilde L)/ u^{i+1}_{\delta+1}(\tilde L), \\
\nonumber x \,\frac{d\alpha_1}{\alpha_1}\wedge \dots \wedge \frac{d\alpha_\delta}{\alpha_\delta}  &\mapsto \{1+\pi^i\tilde x,\tilde \alpha_1,\tilde \alpha_2, \dots, \tilde \alpha_\delta \}; \\
\intertext{for every $0<i<\epsilon_{\tilde L}$ such that $p|i$,}
\label{eq Kato iii bis}
\rho_i : \Omega^{\delta-1}_L/\ker(d) &\to  u^i_{\delta+1}(\tilde L)/ u^{i+1}_{\delta+1}(\tilde L), \\ 
\nonumber x \,\frac{dy_1}{y_1}\wedge \frac{dy_2}{y_2} \wedge \dots \wedge \frac{dy_{\delta-1}}{y_{\delta-1}}  &\mapsto  \{\pi,1+\pi^i\tilde x,\tilde y_1,\tilde y_2, \dots, \tilde y_{\delta-1} \}; \\
\intertext{for $i=\epsilon_{\tilde L}$,}
\label{eq Kato iv bis}
\rho_{i} : H_p^{\delta}(L) &\to u^{i}_{\delta+1}(\tilde L), \\ 
\nonumber x \,\frac{dy_1}{y_1}\wedge \frac{dy_2}{y_2} \wedge \dots \wedge \frac{dy_{\delta-1}}{y_{\delta-1}}  &\mapsto  \{\pi,1+(\zeta_p-1)^p\tilde x,\tilde y_1,\tilde y_2, \dots, \tilde y_{\delta-1} \}.
\end{align}

\subsubsection{Proof of Proposition \ref{proposition cle}}

We prove first the proposition in the particular case where
\begin{equation}\label{eqn a especifico}
a=\{1+{\pi}^{\ell_0}\tilde d_0,\tilde\alpha_1,\dots,\tilde\alpha_\delta\},
\end{equation}
with $ v_{\tilde L}(p)< \ell_0 < \epsilon_{\tilde L}$, $(\ell_0,p)=1$ and $\tilde d_0\in \mathcal{O}_{\tilde L}^\times$. We start with the following two lemmas:

\begin{lemma}\label{lem prod suma}
Let $\tilde K$ be a complete discrete valuation field of characteristic $0$ whose residue field $K$ has characteristic $p>0$. Let $\ell>v_{\tilde K}(p)$ be an integer, let $\pi$ be a uniformizer in $\tilde K$ and let $\tilde c,\tilde d\in\cal O_{\tilde K}$. Then we have the following equality in $\tilde K^\times/(\tilde K^\times)^p$:
\[(1+\pi^\ell\tilde c)(1+\pi^\ell\tilde d)=1+\pi^\ell (\tilde c+\tilde d).\]
\end{lemma}

\begin{proof}
There exists $\tilde x \in \mathcal{O}_{\tilde K}$ such that:
$$(1+\pi^\ell\tilde c)(1+\pi^\ell\tilde d)=(1+\pi^\ell (\tilde c+\tilde d))(1+\pi^{2\ell} \tilde x).$$
Now, since $\ell>v_{\tilde K}(p)$, we have $2\ell>\epsilon_{\tilde K}$. Hence, isomorphism \eqref{eq Kato iv} from Section \ref{sec K-th} applied to $q=0$ tells us that $(1+\pi^{2\ell} \tilde x)$ is a $p$-th power in $\tilde K$.
\end{proof}

\begin{lemma}\label{lem ram}
Let $\tilde K$ be a complete discrete valuation field of characteristic $0$ whose residue field $K$ has characteristic $p>0$. Let $\ell<\epsilon_{\tilde K}$ be an integer prime to $p$, let $\pi$ be a uniformizer in $\tilde K$ and let $\tilde d$ be a unit in $\tilde K$. Then the extension $\tilde K(\sqrt[p]{1+\pi^\ell \tilde d})/\tilde K$ is totally ramified of degree $p$.
\end{lemma}

\begin{proof}
Set $z:=\sqrt[p]{1+\pi^\ell \tilde d}-1\in \tilde K$. It is a zero of the polynomial:
$$\mu:=(X+1)^p-1-\pi^\ell \tilde d\in \tilde K[X].$$
By computing its Newton polygon (cf.~for instance \cite[II, Prop.~6.3]{Neukirch}) and by using the inequality $\ell<\epsilon_{\tilde K}$, one sees that the roots of $\mu$ have valuation $\ell/p$. Since $\ell$ is prime to $p$, this shows that $\tilde K(\sqrt[p]{1+\pi^\ell \tilde d})/\tilde K$ is totally ramified of degree $p$, as wished.
\end{proof}

\begin{proof}[Proof of Proposition \ref{proposition cle} for the symbol \eqref{eqn a especifico}]
Let $H$ be the Galois group of $\tilde L/\tilde K$ (that can also be seen as the Galois group of $L/K$) and let $n$ be its order. Consider $\mathbb{F}_p$ as a subfield of $L$. For each nonzero family $\mathbf{m}:=(m_\sigma)_{\sigma \in H} \in \mathbb{F}_p^H$, consider the sets:
\begin{gather*}
    X_{\mathbf{m}}:=\left\{\lambda \in L : \sum_{\sigma \in H}m_\sigma \sigma(\lambda)^p=0\right\}=\left\{\lambda \in L : \sum_{\sigma \in H}m_\sigma\sigma(\lambda)=0\right\},\\
    Y_{\mathbf{m}}:=\left\{\lambda \in L : \sum_{\sigma \in H}m_\sigma \sigma(\lambda)^p=-\sum_{\sigma \in H}m_\sigma \sigma(d_0)\right\}.
\end{gather*}
Since the elements of $\mathrm{Gal}(L/K)$ are $K$-linearly independent, the sets $X_\mathbf{m}$ are all strict sub-$K$-vector spaces. Moreover, each $Y_{\mathbf{m}}$ is either empty or of the form $y_{\mathbf{m}}+X_{\mathbf{m}}$ for some $y_{\mathbf{m}}\in Y_{\mathbf{m}}$. The field $K$ being infinite, we deduce that:
$$Y:=\bigcup_{\substack{\mathbf{m}\in \mathbb{F}_p^H\\ \mathbf{m}\neq 0}} Y_\mathbf{m} \neq L.$$
In particular, we may and do fix an element $\lambda_0$ in the complement of $Y$ in $L$, as well as a lift $\tilde \lambda_0$ of $\lambda_0$ in $\tilde L$.

Set now $\tilde d_1:=\tilde d_0 + \tilde\lambda_0^p$ and for each $\sigma \in H$ set $y_\sigma:=\sqrt[p]{1+{\pi}^{\ell_0}\sigma(\tilde d_1)}$. Introduce also $x:=\sqrt[p]{{\pi}^{\ell_0}\tilde\alpha_1}$, consider the fields:
$$\tilde L' := \tilde L \left( (y_\sigma)_{\sigma \in H}\right),\;\;\;\;\;\;\;\;
        \tilde L'' := \tilde L' (x)= \tilde L \left( x,(y_\sigma)_{\sigma \in H}\right),$$
and set:
$$H':=\mathrm{Gal}(\tilde L'/\tilde L), \;\;\;\;\; H'':=\mathrm{Gal}(\tilde L''/\tilde L).$$
We start by proving the following three facts about the fields $\tilde L'$ and $\tilde L''$:

\paragraph{Fact 1: The extension $\tilde L'/\tilde K$ is Galois.}

Observe that, for every $\varphi \in \mathrm{Gal}(\tilde K^\sep/\tilde K)$, there exists $\sigma_0 \in H$ such that $\varphi|_{\tilde L}=\sigma_0$ and hence, for each $\sigma \in H$:
$$\varphi(y_\sigma^p)= \sigma_0(y_\sigma^p)=\sigma_0(1+{\pi}^{\ell_0}\sigma(\tilde d_1))=1+ {\pi}^{\ell_0}\sigma_0(\sigma(\tilde d_1))=y_{\sigma_0\sigma}^p.$$
Since $\tilde K$ contains all $p$-th roots of unity, we deduce that the extension $\tilde L'/\tilde K$ is Galois.

\paragraph{Fact 2: The extension $\tilde L'/\tilde L$ is abelian with Galois group $(\Z/p\Z)^{n}$.} Since $\tilde L$ contains all $p$-th roots of unity and $\tilde L'=\tilde L \left( (y_\sigma)_{\sigma \in H}\right)$, it suffices to prove by Kummer theory that the subgroup of $\tilde L^\times/(\tilde L^\times)^p$ generated by $(y^p_\sigma)_{\sigma\in H}$ has order $p^n$. In other words, it suffices to prove that, for any nonzero $(m_\sigma)_{\sigma\in H} \in \{0,\dots, p-1\}^H$:
$$\prod_{\sigma \in H} y_\sigma^{pm_\sigma}\neq 1 \in \tilde L^\times/(\tilde L^\times)^p.$$
Now, using Lemma \ref{lem prod suma} and the assumption that $\ell_0 > v_{\tilde L}(p)$, we compute in $\tilde L^\times/(\tilde L^\times)^p$:
\begin{align*}
    \prod_{\sigma \in H} y_\sigma^{pm_\sigma} & = 1+ \sum_{\sigma\in H} m_\sigma {\pi}^{\ell_0}\sigma(\tilde d_1) \\
    &= 1+ \pi^{\ell_0}\sum_{\sigma\in H} m_\sigma\sigma(\tilde d_1)\\
    &= 1+ \pi^{\ell_0}\sum_{\sigma\in H} m_\sigma(\sigma(\tilde d_0) +\sigma(\tilde\lambda_0)^p)
\end{align*}
By construction of $\tilde \lambda_0$, we have
\[\sum_{\sigma\in H} m_\sigma(\sigma(d_0) +\sigma(\lambda_0)^p)\neq 0\]
in $L$, and hence we get $\prod_{\sigma \in H} y_\sigma^{pm_\sigma}\neq 1 \in \tilde L^\times/(\tilde L^\times)^p$ thanks to Lemma \ref{lem ram} and the assumption that $\ell_0<\epsilon_{\tilde L}$.

\paragraph{Fact 3: The extension $\tilde L''/\tilde L$ is totally ramified.} Since $x^p$ has valuation prime to $p$, Fact 2 easily implies that the subgroup of $\tilde L^\times/(\tilde L^\times)^p$ generated by $x^p$ and by $(y^p_\sigma)_{\sigma\in H}$ has order $p^{n+1}$. Hence the extension $\tilde L''/\tilde L$ is abelian with Galois group $(\Z/p\Z)^{n+1}$. Its degree $p$ subextensions are given by the:
$$\tilde L_{\mathbf{m},r} := \tilde L\left( x^{r}\prod_{\sigma \in H} y_\sigma^{m_\sigma} \right) $$
 for $\mathbf{m}:=(m_\sigma) \in \{0,\dots, p-1\}^H$ and $r\in \{0,\dots, p-1\}$ (with $\mathbf{m}$ and $r$ not both zero). Whenever $r \neq 0$, the valuation of $x^{pr}\prod_{\sigma \in H} y_\sigma^{pm_\sigma}$ in $\tilde L$ is prime to $p$ and hence the extension $\tilde L_{\mathbf{m},r}/\tilde L$ is totally ramified. When $r=0$ and $\mathbf{m}$ is nonzero, the same computations as in Fact 2 show that the product $\prod_{\sigma \in H} y_\sigma^{pm_\sigma}$ belongs to $1+\pi^{\ell_0}\mathcal{O}_{\tilde L}^\times$. Hence, by Lemma \ref{lem ram}, the extension $\tilde L_{\mathbf{m},0}/\tilde L$ is again totally ramified. Summing up, we have proved that the degree $p$ subextensions of the $(\Z/p\Z)^{n+1}$-extension $\tilde L''/\tilde L$ are all totally ramified, and hence $\tilde L''/\tilde L$ is also totally ramified.\\

Back to the proof of Proposition \ref{proposition cle} for the symbol \eqref{eqn a especifico}, Fact 1 allows to introduce the Galois group $G$ of $\tilde L'/\tilde K$, so that we have an exact sequence:
\begin{equation}\label{Galois exact}
1 \rightarrow H' \rightarrow G \rightarrow H \rightarrow 1.    
\end{equation}
Moreover, by Fact 2, we have:
$$H'=\bigoplus_{\sigma \in H} \mathbb{Z}/p\mathbb{Z} \cdot \tau_\sigma,$$ 
where
$$  \tau_\sigma  : y_\rho \mapsto \begin{cases}
        \zeta_p y_\rho \;\; \text{if $\rho=\sigma$}\\
        y_\rho \;\; \text{otherwise.}
    \end{cases} $$
Hence, for each $\sigma,\sigma'\in H$, if we denote by $\tilde \sigma'$ a lifting of $\sigma$ to $G$, we get:
$$\tilde\sigma' \tau_\sigma \tilde \sigma'^{-1}=\tau_{\sigma'\sigma}.$$
In particular, we have $H'=\mathbb{Z}/p\mathbb{Z}[H]$. But by Shapiro's Lemma, $H^2(H,\mathbb{Z}/p\mathbb{Z}[H])=0$. This means by \cite[Thm.~1.2.4]{NSW} that exact sequence \eqref{Galois exact} has a splitting $s: H \rightarrow G$. We introduce the fields:
$$    \tilde K' := \tilde L'^{s(H)},\;\;\;\;\;\;\;\;
        \tilde K'' := \tilde K' (x).$$
By Fact 3, the extension $\tilde L''/\tilde L$ is totally ramified. Hence so is the extension $\tilde K''/\tilde K$. We conclude by using once again Lemma \ref{lem prod suma} to observe that:
\begin{align*}
    a &= \{1+{\pi}^{\ell_0}\tilde d_0,\tilde\alpha_1,\dots,\tilde\alpha_\delta\}\\
    & = \{1+{\pi}^{\ell_0}(\tilde d_1-\tilde\lambda_0^p),\tilde\alpha_1,\dots,\tilde\alpha_\delta\}\\
    & = \{1+{\pi}^{\ell_0}\tilde d_1,\tilde\alpha_1,\dots,\tilde\alpha_\delta\}+\{1-{\pi}^{\ell_0}\tilde\lambda_0^p,\tilde\alpha_1,\dots,\tilde\alpha_\delta\}\\
    &= \{1+{\pi}^{\ell_0}\tilde d_1,\tilde\alpha_1,\dots,\tilde\alpha_\delta\}+\{1-{\pi}^{\ell_0}\tilde\lambda_0^p,{\pi}^{\ell_0}\tilde\lambda_0^p\tilde\alpha_1,\alpha_2,\dots,\tilde\alpha_\delta\}\\
    &= \{1+{\pi}^{\ell_0}\tilde d_1,\tilde\alpha_1,\dots,\tilde\alpha_\delta\}+\{1-{\pi}^{\ell_0}\tilde\lambda_0^p,{\pi}^{\ell_0}\tilde\alpha_1,\alpha_2,\dots,\tilde\alpha_\delta\}
\end{align*}
and hence that $a|_{\tilde L \tilde K''}=a|_{\tilde L''}=0$.
\end{proof}

In order to prove Proposition \ref{proposition cle} in the general case, we need some explicit computations, summarized in the following lemma.

\begin{lemma}\label{lem calculitos}
Keep the hypotheses and notation of Proposition \ref{proposition cle} and Preliminaries \ref{subsection prel proofs}.
\begin{itemize}
\item[(i)] Consider a symbol
$$b_1=\{1+\pi^j \tilde c,\tilde \alpha_1,\tilde \alpha_2,\dots,\tilde \alpha_\delta \}\in \kk_{\delta+1}(\tilde L),$$
where $\tilde c$ is a unit in $\tilde L$ and $j>0$. Then $b_1$ can be written as a sum of symbols of the form
\[\{1+\pi^i \tilde\lambda^p\tilde\alpha_1^{j_1} \dots \tilde\alpha_\delta^{j_\delta},\tilde \alpha_1,\tilde \alpha_2,\dots,\tilde \alpha_\delta\},\]
where $0\leq j_1,\dots,j_\delta<p$, $\tilde\lambda\in\tilde L$ is a unit and $j\leq i<\epsilon_{\tilde L}$.
\item[(ii)] Consider a symbol
\[b_2=\{1+\pi^i \tilde\lambda^p\tilde\alpha_1^{j_1} \dots \tilde\alpha_\delta^{j_\delta},\tilde \alpha_1,\tilde \alpha_2,\dots,\tilde \alpha_\delta\}\in \kk_{\delta+1}(\tilde L),\]
where $0\leq j_1,\dots,j_\delta<p$, $\tilde\lambda\in\tilde L$ is a unit and $0< i<\epsilon_{\tilde L}$. Assume that $p|i$. Then:
\begin{itemize}
\item if $(j_1,\ldots,j_\delta)\neq (0,\ldots,0)$, then $b_2$ is trivial;
\item if $(j_1,\ldots,j_\delta)=(0,\ldots,0)$, then $b_2$ is equal to a symbol of the form
\[\{1+{\pi}^{\ell}\tilde d,\tilde\alpha_1,\dots,\tilde\alpha_\delta\},\]
for some $\ell\geq v_{\tilde L}(p)+\frac{i}{p}$ and some unit $\tilde d$ in $\tilde L$.
\end{itemize}
\end{itemize}
\end{lemma}

\begin{proof}
We prove (i). Since $\{\alpha_1,\alpha_2,\dots,\alpha_\delta\}$ contains a $p$-basis of the residue field $L$, we can write
\[\tilde c\equiv \sum_{0\leq j_1,\dots,j_\delta<p}\tilde\lambda_{j_1,\dots,j_\delta}^p\tilde\alpha_1^{j_1} \dots \tilde\alpha_\delta^{j_\delta} \mod \pi,\]
for some $\tilde\lambda_{j_1,\dots,j_\delta}\in\mathcal{O}_{\tilde L}^\times \cup \{0\}$. This implies that we may always write
\[1+\pi^j \tilde c=\left(\prod_{0\leq j_1,\dots,j_\delta<p}(1+\pi^j \tilde\lambda_{j_1,\dots,j_\delta}^p\tilde\alpha_1^{j_1} \dots \tilde\alpha_\delta^{j_\delta})\right)\cdot (1+\pi^{j'} \tilde d),\]
for some $j'>j$ and $\tilde d$ a unit in $\tilde L$. In particular, we get
$$b_1=\sum_{0\leq j_1,\dots,j_\delta<p}\{1+\pi^j \tilde\lambda_{j_1,\dots,j_\delta}^p\tilde\alpha_1^{j_1} \dots \tilde\alpha_\delta^{j_\delta},\tilde \alpha_1,\tilde \alpha_2,\dots,\tilde \alpha_\delta\}
    +\{1+\pi^{j'} \tilde d,\tilde \alpha_1,\tilde \alpha_2,\dots,\tilde \alpha_\delta\}.$$
Iterating this process, we see that we can always write $b_1$ as a sum of a symbol $b'_1$ of the form:
$$\{1+\pi^{\epsilon} \tilde d_{\epsilon},\tilde \alpha_1,\tilde \alpha_2,\dots,\tilde \alpha_\delta\}$$
with $\epsilon=\epsilon_{\tilde L}$ and $\tilde d_{\epsilon}$ an integer in $\tilde L$, and of symbols of the form
\[\{1+\pi^i \tilde\lambda^p\tilde\alpha_1^{j_1} \dots \tilde\alpha_\delta^{j_\delta},\tilde \alpha_1,\tilde \alpha_2,\dots,\tilde \alpha_\delta\},\]
where $0\leq j_1,\dots,j_\delta<p$, $\tilde\lambda\in\tilde L$ is a unit and $j\leq i<\epsilon_{\tilde L}$. We conclude by observing that $b'_1$ lies in $u_{\delta+1}^{\epsilon}(\tilde L)$ and that it corresponds by isomorphism \eqref{eq Kato iv} to an element of $H^{\delta+1}_p(L)$, which is trivial.\\

We now prove (ii). Assume first that $(j_1,\dots, j_\delta)\neq (0,\dots,0)$. Without loss of generality, we may and do assume that $j_1 \neq 0$. For arbitrary $\ell\in\mathbb{N}$, we have
\begin{align*}
b_2 &= \{1+{\pi}^{i}\tilde\lambda^p\tilde\alpha_1^{j_1}\cdots\tilde\alpha_\delta^{j_\delta}, (-{\pi}^{i}\tilde\lambda^p\tilde\alpha_1^{j_1}\cdots\tilde\alpha_\delta^{j_\delta})^{\ell}\tilde\alpha_1,\tilde\alpha_2,\ldots,\tilde\alpha_\delta\},\\
&=\{1+{\pi}^{i}\tilde\lambda^p\tilde\alpha_1^{j_1}\cdots\tilde\alpha_\delta^{j_\delta}, (-1)^\ell{\pi}^{i\ell}\tilde\lambda^{p\ell}\tilde\alpha_1^{\ell j_1+1},\tilde\alpha_2,\ldots,\tilde\alpha_\delta\}\\
&\hspace{3em}+\sum_{k=2}^\delta \ell j_k\{1+{\pi}^{i}\tilde\lambda^p\tilde\alpha_1^{j_1}\cdots\tilde\alpha_\delta^{j_\delta},\tilde\alpha_k,\tilde\alpha_1,\tilde\alpha_2,\ldots,\tilde\alpha_\delta\}\\
&=\{1+{\pi}^{i}\tilde\lambda^p\tilde\alpha_1^{j_1}\cdots\tilde\alpha_\delta^{j_\delta}, (-1)^\ell{\pi}^{i\ell}\tilde\lambda^{p\ell}\tilde\alpha_1^{\ell j_1+1},\tilde\alpha_2,\ldots,\tilde\alpha_\delta\}
\end{align*}
and since $-1$ is a $p$-th power in $\tilde L$ and $p|i$, the second entry in the last symbol is a $p$-th power provided that $\ell j_1\equiv -1\pmod{p}$. Hence $b_2$ is trivial in $\kk_{\delta+1}(\tilde L)$.

Assume now that $(j_1,\dots, j_\delta)=(0,\dots,0)$. In this case,
\[1+{\pi}^{i}\tilde\lambda^p=(1+{\pi}^{i/p}\tilde\lambda)^p(1+p{\pi}^{i/p}\tilde x),\]
for some $\tilde x\in\tilde L$ with $v_{\tilde L}(\tilde x)\geq 0$. Hence, in $\kk_{\delta+1}(\tilde L)$,
\[b_2=\{1+p{\pi}^{i/p}\tilde x,\tilde\alpha_1,\dots,\tilde\alpha_\delta\}.\]
\end{proof}

\begin{proof}[Proof of Proposition \ref{proposition cle}]
From the isomorphisms \eqref{eq Kato i bis} to \eqref{eq Kato iv bis}, we deduce that the element $a\in\kk_{\delta+1}(\tilde L)$ can always be written as a sum of symbols of the forms
\begin{enumerate}
\item[(a)] $\{\pi,\tilde y_1,\tilde y_2,\dots,\tilde y_\delta \}$, where $\tilde y_1,\tilde y_2\dots,\tilde y_\delta $ are units in $\tilde L$; and 
\item[(b)] $\{1+\pi^i \tilde c,\tilde \alpha_1,\tilde \alpha_2,\dots,\tilde \alpha_\delta \}$, where $\tilde c$ is a unit in $\tilde L$ and $0< i < \epsilon_{\tilde L}$ with $(i,p)=1$.
\end{enumerate}

By applying Lemma \ref{lem calculitos}.(i) to the terms of type (b) that appear in $a$, we see that $a$ can always be written as a sum of symbols of the form (a) and of the form
\[\{1+\pi^i \tilde\lambda^p\tilde\alpha_1^{j_1} \dots \tilde\alpha_\delta^{j_\delta},\tilde \alpha_1,\tilde \alpha_2,\dots,\tilde \alpha_\delta\},\]
where $0\leq j_1,\dots,j_\delta<p$, $\lambda\in\tilde L$ is a unit and $0< i<\epsilon_{\tilde L}$. Observe now that the finite extension $\tilde L':=\tilde L(\sqrt[p]{\pi})$ of $\tilde L$ trivializes any symbol of the form (a), so the restriction of $a$ to $\tilde L'$ is a sum of symbols of the form
\[\{1+{\pi'}^{pi} \tilde\lambda^p\tilde\alpha_1^{j_1} \dots \tilde\alpha_\delta^{j_\delta},\tilde \alpha_1,\tilde \alpha_2,\dots,\tilde \alpha_\delta\},\]
where $0\leq j_1,\dots,j_\delta<p$, $\tilde\lambda\in\tilde L$ is a unit, $0< i<\epsilon_{\tilde L}$ and $\pi':=\sqrt[p]{\pi}$. We may then apply Lemma \ref{lem calculitos}.(ii) to these and deduce that the restriction of $a$ to $\tilde L'$ is a sum of symbols of the form
\[\{1+{\pi'}^{\ell}\tilde d,\tilde\alpha_1,\dots,\tilde\alpha_\delta\},\]
for some $\ell> v_{\tilde L'}(p)$ and some unit $\tilde d$ in $\tilde L'$. And since all these symbols share the components $\tilde\alpha_1,\ldots,\tilde\alpha_\delta$, we may put them together and finally conclude that 
\[a|_{\tilde L'}=\{1+{\pi'}^{\ell_0}\tilde d_0,\tilde\alpha_1,\dots,\tilde\alpha_\delta\},\]
for some $\ell_0> v_{\tilde L'}(p)$ and some $\tilde d_0 \in \tilde L'$.

Observe now that, if $p|\ell_0$, we may apply successively parts (i) and (ii) of Lemma \ref{lem calculitos} to this symbol and obtain that $a|_{\tilde L'}$ is of the form:
\[\{1+{\pi'}^{\ell'_0}\tilde d'_0,\tilde\alpha_1,\dots,\tilde\alpha_\delta\},\]
for some $\ell'_0\geq v_{\tilde L'}(p)+\frac{\ell_0}{p}$. But, if $\ell_0<\epsilon_{\tilde L'}=v_{\tilde L'}(p)p/(p-1)$, then $\ell'_0>\ell_0$. Iterating this argument, we may therefore assume that either $(\ell_0,p)=1$ or $\ell_0\geq \epsilon_{\tilde L'}$. But in the latter case, the isomorphism \eqref{eq Kato iv} and the fact that $H_p^{q+1}(L)=0$ immediately imply that $a|_{\tilde L'}=0$.

We are hence reduced to the case where $v_{\tilde L'}(p)<\ell_0<\epsilon_{\tilde L'}$ and $(\ell_0,p)=1$, so that, up to replacing $\tilde K$ by $\tilde K (\sqrt[p]{\pi})$ and $\tilde L$ by $\tilde L (\sqrt[p]{\pi})=\tilde L (\sqrt[p]{\pi})$, we only need to prove the proposition for
$$a=\{1+{\pi}^{\ell_0}\tilde d_0,\tilde\alpha_1,\dots,\tilde\alpha_\delta\},$$
which is the case we have already dealt with.
\end{proof}

\subsubsection{Proof of Theorem \ref{thm car p vers car 0}}

Let $p$ be the characteristic exponent of $K$ and let $\pi\in\tilde K$ be a uniformizer. The case of a finite residue field is easy to settle. Thus, if $p\neq 1$, up to replacing $\tilde K$ by $\tilde K(\zeta_p)$ (and by $\tilde K(\sqrt{-1})$ if $p=2$), we may assume that $\tilde K$ satisfies the hypotheses of Proposition \ref{proposition cle}.

All extensions of $\tilde K$ considered in this proof, in particular composite fields, will be assumed to be subfields of a fixed algebraic closure of $\tilde K$.\\

Consider a compatible system $(\sqrt[n]{\pi})_{n\geq 1}$ of $n$-th roots of $\pi$ in the fixed algebraic closure of $\tilde{K}$, set $\tilde{K}_n:=\tilde{K}(\sqrt[n]{\pi})$ for each $n \geq 1$, and introduce the field
$$\tilde{K}_{(p')}:= \bigcup_{(n,p)=1} \tilde{K}_n,$$
which is a totally ramified extension of $\tilde K$.

We claim that $\mathrm{cd}_\ell(\tilde{K}_{(p')})\leq \delta$ for every prime $\ell\neq p$. Fix such a prime number $\ell$ and a finite extension $\tilde{L}$ of $\tilde{K}_{(p')}$. There exists an integer $n_0$ prime to $p$ and a finite extension $\tilde{L}_{n_0}$ of $\tilde{K}_{n_0}$ such that  $\tilde{L}=\tilde{L}_{n_0}\tilde{K}_{(p')}$. For each $n \geq n_0$, we set $\tilde{L}_n:=\tilde{L}_{n_0}(\sqrt[n]{\pi})$ and we denote by $L_n$ the residue field of $\tilde{L}_n$. Note that, up to increasing $n_0$, we may assume that $\tilde L_n\cap \tilde K_{(p')}=\tilde K_n$ for every $n\geq n_0$. Then, for each $n \geq n_0$ and $m\geq 1$, the residue maps induce the following commutative diagram:
$$
\xymatrix{
H^{\delta+1}(\tilde{L}_{mn},\mu_\ell^{\otimes (\delta+1)})  \ar[r]^-{\partial}_-\cong & H^\delta(L_{mn},\mu_\ell^{\otimes \delta}) \\
H^{\delta+1}(\tilde{L}_n,\mu_\ell^{\otimes (\delta+1)}) \ar[u]^{\mathrm{Res}_{\tilde{L}_{mn}/\tilde{L}_{n}}} \ar[r]^-{\partial}_-\cong & H^\delta(L_n,\mu_\ell^{\otimes \delta}) \ar[u]_{e_{\tilde{L}_{mn}/\tilde{L}_{n}}\cdot\mathrm{Res}_{L_{mn}/L_{n}}}}
$$
in which the horizontal lines are isomorphisms since $L_r$ has cohomological dimension $\delta$ for every $r$ (cf.~\cite[II.~App.~\S2]{SerreCohGal} and \cite[Prop.~3.3.1]{ColliotSantaBarbara}). And since $\tilde L_n\cap \tilde K_{(p')}=\tilde K(\sqrt[n]{\pi})$, the integer $e_{\tilde{L}_{mn}/\tilde{L}_{n}}$ is always divisible by $m$. Hence, the group
$$H^{\delta+1}(\tilde{L}_{(p')},\mu_\ell^{\otimes (\delta+1)}) =\varinjlim_n H^{\delta+1}(\tilde{L}_{n},\mu_\ell^{\otimes (\delta+1)}),$$
is trivial. This being true for every finite extension $\tilde{L}$ of $\tilde{K}_{(p')}$, we get $\mathrm{cd}_\ell(\tilde{K}_{(p')})\leq \delta$.\\

This deals with the case where $K$ has characteristic $0$, so we assume hereafter that $K$ has characteristic $p>0$. We construct now a second extension $\tilde K_{(p)}$ of $\tilde K$ as follows. We proceed as in Proposition \ref{prop unc}. In particular, we fix the following notations:
\begin{itemize}
\item We let $\sigma=(\sigma_1,\sigma_2): \mathbb{N} \rightarrow \mathbb{N}^2$ be a bijection such that $\sigma_1(n)\leq n$ for each $n \in \mathbb{N}$.
\item For each finite extension $\tilde L$ of $\tilde K$, we define $\mathcal{E}_{\tilde L}$ as the set of pairs $(\tilde M,a)$ such that $\tilde M$ is a finite, tamely ramified, Galois extension of $\tilde L$ with separable residue field extension $M/L$, and $a$ is an element in $\kk_{\delta+1}(\tilde M)=H^{\delta+1}(\tilde M,\mu_p^{\otimes (\delta+1)})$.
\end{itemize}
We claim that, for each finite extension $\tilde L$ of $\tilde K$, the set $\mathcal{E}_{\tilde L}$ is countable. Indeed, since $K$ is countable, so is the set of unramified extensions of $\tilde L$. Moreover, since every tamely ramified extension of the maximal unramified extension of $\tilde L$ consists in taking a root of a fixed uniformizer, we see that tamely ramified extensions of $\tilde L$ are countable as well. Furthermore, for each such extension $\tilde M/\tilde L$, the group $\kk_{\delta+1}(\tilde M)$ is countable, according to the isomorphisms \eqref{eq Kato i} to \eqref{eq Kato iv}. Thus, we may assume that $\mathcal E_{\tilde L}$ comes with a given numbering.\\

Let us now inductively construct an increasing infinite sequence of finite and totally ramified extensions $(\tilde L_i)_{i\geq 1}$ of $\tilde K$ satisfying the following property: the maximal tamely ramified extension $\tilde L^{\mathrm{tam}}_i$ of $\tilde K$ in $\tilde L_i$ is of the form $\tilde K(\sqrt[e_i]{\pi})$ for some $e_i$ prime to $p$. To do so, we set first $\tilde L_1:=\tilde K$. Then, for $i \geq 1$, consider the $\sigma_2(i)$-th term of the set $\mathcal E_{L_{\sigma_1(i)}}$, which we denote by $(\tilde M_i,a_i)$. Of course, since $\sigma_1(i)\leq i$, we have $\tilde L_{\sigma_1(i)}\subset \tilde L_{i}$.

Since $\tilde M_i/\tilde L_i$ is tamely ramified, its ramification degree $e'_i$ is prime to $p$. Set now $e_{i+1}:=e_ie_i'$. By the inductive assumption, $\tilde L^{\mathrm{tam}}_i=\tilde K(\sqrt[e_{i}]{\pi})$, and hence the extension $\tilde L_i(\sqrt[e_{i+1}]{\pi})/\tilde K$ is totally ramified while the extension $\tilde M_i (\sqrt[e_{i+1}]{\pi})/\tilde L_i(\sqrt[e_{i+1}]{\pi})$ is unramified. Applying Proposition \ref{proposition cle} to the unramified extension $\tilde M_i (\sqrt[e_{i+1}]{\pi})/\tilde L_i(\sqrt[e_{i+1}]{\pi})$ and to the symbol $a_i|_{\tilde M_i (\sqrt[e_{i+1}]{\pi})}$, we know that there exists a finite and totally ramified extension $\tilde L_{i+1}/\tilde L_i(\sqrt[e_{i+1}]{\pi})$ of $p$-primary degree such that $a_i|_{\tilde L_{i+1}\tilde M_i}=0$. By construction, we have $\tilde L^{\mathrm{tam}}_{i+1}=\tilde K(\sqrt[e_{i+1}]{\pi})$. We summarize this in Figure \ref{fig field tower}, where we input the \emph{ramification degrees} for each extension.\\

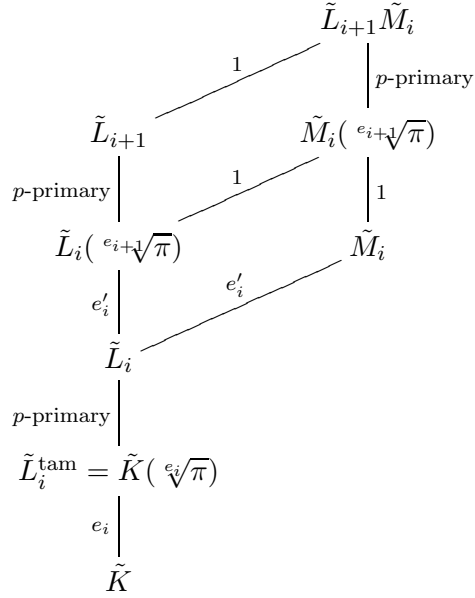
\begin{figure}
\[\xymatrix{
& \tilde L_{i+1}\tilde M_i \ar@{-}[dl]_{1} \ar@{-}[d]^{p\text{-primary}} \\
\tilde L_{i+1} \ar@{-}[d]_{p\text{-primary}} & \tilde M_i(\sqrt[e_{i+1}]{\pi}) \ar@{-}[d]^{1} \ar@{-}[dl]_{1}\\
\tilde L_i(\sqrt[e_{i+1}]{\pi}) \ar@{-}[d]_{e_i'} & \tilde M_i \ar@{-}[dl]_{e_i'} \\
\tilde L_i \ar@{-}[d]_{p\text{-primary}} & \\
\tilde L_i^{\mathrm{tam}}=\tilde K(\sqrt[e_i]{\pi}) \ar@{-}[d]_{e_i} & \\
\tilde K
}\]
\caption{Extensions involved in the construction of $\tilde K_{(p)}$ and their corresponding ramification degrees.}\label{fig field tower}
\end{figure}

Let us now introduce the field:
$$\tilde{K}_{(p)}:= \bigcup_{i\geq 1} \tilde L_i.$$
Observe that the extension $\tilde{K}_{(p)}/\tilde{K}$ is totally ramified and that the maximal tamely ramified extension of $\tilde{K}$ contained in $\tilde{K}_{(p)}$ is contained in $\tilde{K}_{(p')}$. Hence the composite $\tilde{K}_\dagger:=\tilde{K}_{(p)}\tilde{K}_{(p')}$ is totally ramified over $\tilde{K}$.

Recall that, for each prime number $\ell$ other than $p$, we have $\mathrm{cd}_\ell(\tilde{K}_{(p')})\leq \delta$, and hence $\mathrm{cd}_\ell(\tilde{K}_{\dagger})\leq \delta$. To conclude that $\mathrm{cd}(\tilde{K}_{\dagger})\leq \delta$, it therefore suffices to check that $\cd_p(\tilde{K}_{(p)})\leq \delta$. But according to Lemma \ref{lem coho dim} below, that amounts to check that the group $H^{\delta+1}(\tilde M_{(p)},\mu_p^{\otimes (\delta+1)})$ is trivial for each finite tamely ramified Galois extension $\tilde M_{(p)}/\tilde{K}_{(p)}$ with separable residue field extension. We henceforth fix an element $a \in H^{\delta+1}(\tilde M_{(p)},\mu_p^{\otimes (\delta+1)})$. We can then find an integer $i \geq 1$, a finite tamely ramified Galois extension $\tilde M$ of $\tilde L_i$ with separable residue field extension and an element $b\in H^{\delta+1}(\tilde M,\mu_p^{\otimes (\delta+1)})$ such that $\tilde M \tilde{K}_{(p)}=\tilde M_{(p)}$ and $b|_{\tilde M_{(p)}}=a$. Now, the pair $(\tilde M,b)$ is the $j$-th element of $\mathcal E_{\tilde L_i}$ for some $j\in\mathbb N$. This implies that $b|_{\tilde M\tilde L_{\sigma^{-1}(i,j)+1}}=0$ by construction. By the inclusion $\tilde M \tilde L_{\sigma^{-1}(i,j)+1}\subset \tilde M\tilde{K}_{(p)}=\tilde M_{(p)}$, we deduce that $a=b|_{\tilde M_{(p)}}=0$, as wished. \qed

\begin{lemma}\label{lem coho dim}
Let $\tilde K$ be a complete discrete valuation field of characteristic $0$ whose residue field $K$ has characteristic $p>0$. Let $\delta$ be the cohomological dimension of $K$ and let $\tilde L$ be an algebraic extension of $\tilde K$ such that, for each finite tamely ramified Galois extension $\tilde M$ of $\tilde L$ with separable residue field extension $M/L$, the cohomology group $H^{\delta+1}(\tilde M,\mu_p^{\otimes (\delta+1)})$ vanishes. Then $\cd_p(\tilde L)\leq \delta$.
\end{lemma}

\begin{proof}
According to Proposition \ref{prop Serre}, it suffices to prove that $H^{\delta+1}(\tilde N,\mu_p^{\otimes (\delta+1)})=0$ for each finite extension of $\tilde N$ of $\tilde L$ of degree prime to $p$ and containing $\zeta_p$. For that purpose, denote by $\tilde M$ the Galois closure of $\tilde N/\tilde L$ and write the exact sequence of finite Galois modules over $\tilde L$:
$$0 \rightarrow F \rightarrow \mu_p^{\otimes (\delta+1)}[\tilde M/\tilde L] \rightarrow \mu_p^{\otimes \delta+1}[\tilde N/\tilde L] \rightarrow 0.$$
It induces an exact sequence of cohomology groups:
$$H^{\delta+1}(\tilde M,\mu_p^{\otimes (\delta+1)}) \rightarrow H^{\delta+1}(\tilde N,\mu_p^{\otimes (\delta+1)}) \rightarrow H^{\delta+2}(\tilde L,F).$$
Note now that the maximal tamely ramified extension (with separable residue field extension) $\tilde K^{\mathrm{tr}}/\tilde K$ is Galois. Moreover, since $\tilde N/\tilde L$ has degree prime to $p$, it is tamely ramified and has separable residue field extension. Thus, the extension $\tilde M/\tilde L$ has the same two properties, and is Galois, and hence $H^{\delta+1}(\tilde M,\mu_p^{\otimes (\delta+1)})=0$. Moreover, by \cite[Cor.~to Thm.~3]{Kato}, we have $\cd(\tilde K)= \delta+1$, so that $H^{\delta+2}(\tilde L,F)=0$. We deduce that $H^{\delta+1}(\tilde N,\mu_p^{\otimes (\delta+1)})=0$.
\end{proof}

\subsection{From higher to lower cohomological dimension}

Our final transfer principle for the cohomological dimension of fields involves the notion of a universal norm:

\begin{definition}
Let $L/K$ be an algebraic field extension and let $x \in K^\times$. We say that $x$ is a \emph{universal norm} for $L/K$ if $x\in N_{K'/K}(K'^\times)$ for every finite extension $K'$ of $K$ contained in $L$.
\end{definition}

\begin{theorem}\label{th countable}
Let $\delta\geq 1$ be an integer, $\ell$ a prime number and $K$ an $\ell$-special countable field of characteristic $0$ and with cohomological dimension $\delta$. For each $x \in K^\times$, there exists an algebraic extension $K_x$ of $K$ that has cohomological dimension $\leq (\delta-1)$ and for which $x$ is a universal norm.
\end{theorem}

\begin{proof}
If $\delta=1$, we take $K_x$ to be the algebraic closure of $K$ and then the statement follows from the surjectivity of the norm $N_{L/K}:L^\times \to K^\times$ for every finite extension $L/K$, which is a well-known fact (cf.~\cite[Chapter II, \S3.1, Prop.~5]{SerreCohGal}).

From now on, we assume $\delta\geq 2$. Let $\mathcal{E}_K$ be the set of pairs $(L,a)$ such that $L$ is a finite extension of $K$ and $a$ a nonzero symbol in $H^\delta(L,\mu_\ell^{\otimes \delta})$. Since $K$ is countable, so is the set $\mathcal{E}_K$ and hence we can number its elements: $(L_1,a_1), (L_2,a_2), \dots$

Let us now inductively construct a sequence of pairs $(K_i,x_i)_{i \geq 0}$ in the following way. For $i=0$, we set $K_0=K$ and $x_0=x$. For $i \geq 0$, we let $n_{i}$ be the smallest integer such that $L_{n_{i}}\subset K_{i}$ and $a_{n_{i}}|_{K_{i}}\neq 0$. According to Theorem 1.21(1-2) of \cite{SJ}, there exists a geometrically irreducible projective generic splitting $\nu_{\delta-1}$-variety
$X_{i}$ over $K_{i}$ of dimension $\ell^{\delta-1}-1$ for the symbol $a_{n_{i}}|_{K_{i}}$ (see Definitions 1.10 and 1.20 of \cite{SJ}). Moreover, by Theorem A.1 of \cite{SJ}, we have an exact sequence:
$$\bigoplus_{p \in X_{i}\text{ closed}} K_{i}(p)^\times \xrightarrow{\bigoplus N_{K_{i}(p)/K_{i}}} K_{i}^\times \xrightarrow{a_{n_{i}}|_{K_{i}} \cup\,\cdot} \K_{\delta+1}(K_{i})/\ell. $$
Since $K$ has cohomological dimension $\delta$, we deduce that $x_{i}$ is the image of some element $\alpha_{i} \in \bigoplus_{p \in X_{i}\text{ closed}} K_{i}(p)^\times$ by the morphism $\bigoplus N_{K_{i}(p)/K_{i}}$. But, if we set:
$$A_0(X_{i},\mathcal{K}_1):=\mathrm{coker} \left( \bigoplus_{\substack{q \in X_{i} \\ \dim \overline{\{q\}}=1 } } \K_2(K_{i}(q)) \xrightarrow{\bigoplus\partial} \bigoplus_{p \in X_{i}\text{ closed}} K_{i}(p)^\times \right),$$
the morphism $\bigoplus N_{K_{i}(p)/K_{i}}$ factors through:
$$\overline{A}_0(X_{i},\mathcal{K}_1):=\mathrm{coker} \left( A_0(X_{i}\times X_{i},\mathcal{K}_1) \xrightarrow[]{(p_1)_*-(p_2)_*} A_0(X_{i},\mathcal{K}_1) \right).$$
Moreover, Theorem 1.21(3) of \cite{SJ} shows that there exist a closed point $p_{i}$ of $X_{i}$ and an element $\lambda_{i} \in K_{i}(p_{i})^\times$ such that its image via the inclusion:
$$ K_{i}(p_{i})^\times\hookrightarrow \bigoplus_{p \in X_{i}\text{ closed}} K_{i}(p)^\times,$$
has the same image as $\alpha_{i}$ in $\overline{A}_0(X_{i},\mathcal{K}_1)$. We deduce that:
$$x_{i}=N_{K_{i}(p_{i})/K_{i}}(\lambda_{i}). $$
We then set $K_{i+1}:=K_{i}(p_{i})$ and $x_{i+1}:=\lambda_{i}$. For the sequel, it will be important to note that for each $i \geq 0$, we have the following properties:
\begin{numcases}{}
K_i\subset K_{i+1},\label{case1}\\
x_{i}=N_{K_{i+1}/K_{i}}(x_{i+1}),\label{case2}\\
a_{n_i}|_{K_{i}}\neq 0,\label{case3}\\
a_{n_i}|_{K_{i+1}}=0.\label{case4}
\end{numcases}
In particular, \eqref{case3} and \eqref{case4} imply that the $n_i$'s are pairwise distinct.

Let us now introduce the field:
$$K_x:= \bigcup_{i\geq 0} K_i.$$
We claim that $K_x$ satisfies the conditions of the theorem. 

Indeed, equation \eqref{case2}, together with the equality $x_0=x$, shows that $x$ is a universal norm for $K_x/K$. It remains to show that $\mathrm{cd}(K_x)\leq \delta-1$. By Proposition \ref{prop Serre} and the Bloch-Kato conjecture, it is enough to prove that every nonzero symbol in $ H^\delta(K_x,\mu_\ell^{\otimes \delta})$ is trivial. Let $a$ be such a symbol. We can then find an integer $i \geq 0$ and a symbol $b \in H^\delta(K_i,\mu_\ell^{\otimes \delta})$ such that $b|_{K_x}=a$. Since the pair $(K_i,b)$ belongs to $\mathcal{E}_K$, we can find an integer $N$ such that $(K_i,b)=(L_N,a_N)$. Then we have two cases:
\begin{itemize}
    \item[(a)] If $N=n_{j_0}$ for some $j_0$, then we have $a_{n_{j_0}}|_{K_{j_0+1}}=0$ by equation \eqref{case4}, and hence $a=b|_{K_x}=a_{n_{j_0}}|_{K_x}=0$.
    \item[(b)] If $N\neq n_j$ for all $j$, since the $n_j$'s are pairwise distinct, there exists an integer $j_0>i$ such that $n_{j_0}>N$. Then $L_N=K_i\subset K_{j_0}$ and hence, by definition of $n_{j_0}$, we deduce that $a_N|_{K_{j_0}}=0$. Hence $a=b|_{K_x}=a_{N}|_{K_x}=0$.
\end{itemize}
\end{proof}

\begin{remarque}
The assumption that the field $K$ is $\ell$-special and of characteristic $0$ in Theorem \ref{th countable} is needed in the proof to apply the results of \cite{SJ}.

However, when $\delta=2$, splitting varieties can be made explicit and thus one can replace $K$ in Theorem \ref{th countable} by an \emph{arbitrary} countable field. More precisely, we have the following result (we thank Philippe Gille for asking this question). 
\end{remarque}

\begin{proposition}
Let $K$ be a countable field of cohomological dimension $2$. For each $x \in K^\times$, there exists an algebraic extension $K_x$ of $K$ that has cohomological dimension $\leq 1$ and for which $x$ is a universal norm.
\end{proposition}

\begin{proof}
We argue as in Theorem \ref{th countable}. Define $\cal E_K$ as the set of triples $(L,\ell,a)$ such that $L$ is a finite extension of $K$, $\ell$ is a prime number and $a$ a nonzero element in $\Br(L)[\ell]$. This set is countable, so we can number its elements: $(L_1,\ell_1,a_1)$, $(L_2,\ell_2,a_2)$, \ldots

In the sequel, we construct inductively a sequence of pairs $(K_i,x_i)_{i \geq 0}$ satisfying $(K_0,x_0)=(K,x)$ and properties \eqref{case1} to \eqref{case4} where, for each $i$, the integer $n_i$ stands for the smallest one such that $L_{n_{i}}\subset K_{i}$ and $a_{n_{i}}|_{K_{i}}\neq 0$. The argument from Theorem \ref{th countable} then proves that every finite extension of $K_x:=\bigcup_{i\in\N} K_i$ has a trivial Brauer group and that $x$ is a universal norm for $K_x/K$. Moreover, when $K$ has characteristic $p>0$, our construction ensures that $[K_x:K_x^p]\leq p$, so that $\cd(K_x)\leq 1$.

In order to construct the sequence $(K_i,x_i)_{i \geq 0}$, we follow the proof of Theorem \ref{th countable}, but replacing the splitting variety $X_i$ by the Severi-Brauer variety $Y_i$ associated to $a_{n_{i}}|_{K_{i}}$. More precisely:
\begin{itemize}
\item If the characteristic of $K$ is $0$, since $a_{n_{i}}|_{K_{i}}\cup x_i=0 \in H^3(K_i,\mu_{\ell_{n_i}}^{\otimes 2})$, we have that $x_i$ is a reduced norm for a central simple algebra representing $a_{n_i}$ by \cite[Thm.~8.9.1]{GS} (cf. also \cite[Thm.~24.4]{SuslinNorm}). Then by \cite[Prop.~2.6.8]{GS} there exists a finite extension $K_{i+1}$ of $K_i$ such that $a_{n_{i}}|_{K_{i+1}}=0$ and $x_i=N_{K_{i+1}/K_{i}}(x_{i+1})$ for some $x_{i+1}\in K_{i+1}$.
\item If the characteristic of $K$ is $p>0$, then we need to take more steps:
\begin{itemize}
\item First we consider, if it exists, a maximal $n_0\in\N$ such that $x_i=y_i^{p^{n_0}}$ for some $y_i\in K_i$. Then $y_i$ belongs to a $p$-basis of $K_i$. If there is no such maximal $n_0$, then we take $y_i$ to be any non-trivial element of a $p$-basis of $K_i$. In any case, we define $K_i':=K_i(\sqrt[p]{y_i})$, so that in particular $x_i^{1/p}\in K_i'$.
\item If the characteristic of $K$ is not $\ell_{n_i}$, then we argue exactly as in the characteristic $0$ case in order to construct a finite extension $K_{i+1}/K_i'$ such that $a_{n_{i}}|_{K_{i+1}}=0$ and $x_i^{1/p}=N_{K_{i+1}/K_{i}'}(x_{i+1})$ for some $x_{i+1}\in K_{i+1}$, so that $x_i=N_{K_{i+1}/K_{i}}(x_{i+1})$.
\item If $K$ has characteristic $\ell_{n_i}$, then the argument is the same, but we replace the cup product by the wedge product and \cite[Thm.~8.9.1]{GS} by \cite[Thm.~6]{gille}. In order to apply this last result we also use the fact that every class in $\Br(L_{n_i})[\ell_{n_i}]$ can be represented by a cyclic algebra, cf.~\cite[Thm.~9.1.8]{GS}.
\end{itemize}
\end{itemize}
In order to prove that $[K_x:K_x^p]\leq p$, by Lemma \ref{lema p-base} here below we simply need to notice that $K_x/K$ contains, by construction, a purely inseparable extension of infinite degree.
\end{proof}

\begin{lemma}\label{lema p-base}
Let $K$ be a field of characteristic $p>0$ and assume that $[K:K^p]=p^n$ for some $n \geq 1$. Let $L/K$ be an algebraic extension of infinite inseparable degree. Then $[L:L^p]<p^n$.
\end{lemma}

\begin{proof}
Since $L/K$ is algebraic, $[L:L^p]\leq p^n$. Assume by contradiction that $[L:L^p]= p^n$, and denote by $(\alpha_1,\dots,\alpha_n)$ a $p$-basis of $L$. Then there exists a finite extension $K'$ of $K$ in $L$ such that $(\alpha_1,\dots,\alpha_n)$ is a $p$-basis of $K'$. Since $K'$ and $L$ share a $p$-basis, the extension $L/K'$ is separable, which contradicts the assumption that $L/K$ has infinite inseparable degree. We deduce that $[L:L^p]<p^n$.
\end{proof}

\section{Applications}\label{sec 4}

In this section, we give several applications of the previous transfer principles for Galois cohomology. We start by proving the transfer principle for norm groups given by Theorem \ref{MTD} and then we settle Theorems \ref{MTE} and \ref{MTF} about Serre's conjecture II and its higher versions.

\subsection{A transfer principle for norm groups}

\begin{theorem}\label{th transfer principle car 0}
Let $m,n,q \geq 0$ be three integers with $n \geq m \geq 1$ and $K$ a field of characteristic 0 and with cohomological dimension $\leq n$. Let $Z$ be a $K$-variety and assume that there exists a countable subfield $K_0$ of $K$ and a form $Z_0$ of $Z$ over $K_0$ satisfying the following condition:
\begin{itemize}
\item[$(\star)$] For every countable subextension $L$ of $K/K_0$, and for every algebraic extension $M/L$ of cohomological dimension $\leq m$, we have $N_{q}(Z_{0,M}/M)=\K_{q}(M)$.
\end{itemize}
Then $N_{n-m+q}(Z/K)=\K_{n-m+q}(K)$.
\end{theorem}

\begin{proof}
We prove first the particular case where $K=K_0$. For that purpose, we may assume, by Lemma \ref{lem p Sylow}, that $K$ is $\ell$-special for some prime number $\ell$. We then proceed by induction on $r:=n-m$. The result is obvious when $r=0$ since then $m=n$ and $K$ has cohomological dimension $\leq n$. We henceforth fix some $r_0\geq 1$, assume to have proven the theorem whenever $r<r_0$ and study the case when $r=r_0$.

To do so, consider a symbol $a:=\{a_1,\dots ,a_{r+q}\}\in \K_{r+q}(K)$. Introduce the field $K_{a_{r+q}}$ given by Theorem \ref{th countable}. It has cohomological dimension $\leq n-1$. By the inductive assumption, we have $N_{r+q-1}(Z_{K_{a_{r+q}}}/K_{a_{r+q}})=\K_{r+q-1}(K_{a_{r+q}})$. In particular, $\{a_1,\dots,a_{r+q-1}\}\in N_{r+q-1}(Z_{K'}/K')$ for some finite subextension $K'$ of $K_{a_{r+q}}/K$. This means that there exist finite extensions $K'_1, \ldots, K'_r$ of $K'$ and elements $z_1\in\K_q(K'_1), \ldots,$ $z_r\in\K_q(K'_r)$ such that:
\[\begin{cases} 
\forall\,i, \; Z(K'_i)\neq\emptyset, \\
\{a_1,\dots,a_{r+q-1}\} = \sum_i N_{K'_i/K'}(z_i).
\end{cases}\]
Now, since $a_{r+q}$ is a universal norm in $K_{a_{r+q}}$, we get that $a_{r+q}=N_{K'/K}(z_0)$ for some $z_0\in {K'}^\times$. Thus,
\begin{align*}
a&=\{a_1,\dots,a_{r+q-1},a_{r+q}\} = \{a_1,\dots,a_{r+q-1},N_{K'/K}(z_0)\}=N_{K'/K}(\{a_1,\dots,a_{r+q-1},z_0\})\\
&=N_{K'/K}(\{\sum_i N_{K'_i/K'}(z_i),z_0\})=\sum_i N_{K'/K}(N_{K'_i/K'}(\{ z_i,z_0\}))=\sum_i N_{K'_i/K}(\{ z_i,z_0\}),
\end{align*}
and hence $a \in N_{r+q}(Z/K)$, as wished.\\

Now we move on to the general case. Consider an element $a\in \K_{n-m+q}(K)$. There exists a countable subextension $K'$ of $K/K_0$ and an element $a' \in \K_{n-m+q}(K')$ such that $a'|_{K}=a$. By Proposition \ref{prop unc}, we can find a countable subextension $L$ of $K/K'$ that has cohomological dimension $\leq n$. Moreover, by assumption $(\star)$, we have $N_q(Z_{0,M}/M)=\K_q(M)$ for every algebraic extension $M$ of $L$ with cohomological dimension $\leq m$. By the case $K=K_0$ applied to $L$, we deduce that $N_{n-m+q}(Z_{0,L}/L)=\K_{n-m+q}(L)$. In particular, $a'|_{L}\in N_{n-m+q}(Z_{0,L}/L)$, and hence $a=a'|_{K}\in N_{n-m+q}(Z/K)$.
\end{proof}

A natural application of this transfer principle for norm groups concerns Kato and Kuzumaki's conjectures. To see that, let us recall the definition of the $C_i^q$ properties.

\begin{definition}[Kato-Kuzumaki, \cite{KK}]
Let $K$ be a field and let $i,q\geq 0$ be two integers. The field $K$ is said to have the $C_i^q$ property if, for each $n \geq 1$, for each finite extension $L$ of $K$ and for each hypersurface $Z$ in $\mathbb{P}^n_{L}$ of degree $d$ with $d^i \leq n$, one has $N_q(Z/L)=\mathrm{K}_q(L)$.
\end{definition}

Kato and Kuzumaki conjectured in \cite{KK} that a field has the $C_i^q$ property if, and only if, its cohomological dimension is $\leq q+i$. Even though it is known nowadays that these conjectures are false in general, they are still open questions for fields that appear naturally in arithmetic geometry. Indeed, the only known counterexamples are built by means of transfinite induction \cite{Merkurjev, CTM}. Moreover, some instances of these conjectures are known to hold for number fields and $p$-adic fields \cite{WittenbergKK}, function fields of complex varieties and fields of Laurent series over these function fields \cite{IzquierdoKK}, and function fields of $p$-adic curves \cite{ILAp}.\\

In this context, Theorem \ref{th transfer principle car 0} has the following immediate consequence.

\begin{corollary}\label{cor kk}
Let $i,m,n,q \geq 0$ be four integers with $n \geq m \geq 1$ and $i+q \geq m$. Let $K$ be a field of characteristic 0 and with cohomological dimension $\leq n$. Assume that: 
\begin{itemize}
\item[$(\diamond)$] For every large enough countable subfield $L$ of $K$, every algebraic extension $M/L$ of cohomological dimension $\leq m$ satisfies the $C_i^q$ property.
\end{itemize}
Then $K$ satisfies the $C_i^{n-m+q}$ property.
\end{corollary}

In particular, we get the following two small direct applications:

\begin{corollary}\label{cor KK fgQ}
Let $K$ be a field with finite transcendence degree over $\mathbb{Q}$. If, for every $i\geq 1$, algebraic extensions of $K$ with cohomological dimension $i$ satisfy the $C_i^0$ property, then Kato and Kuzumaki's conjectures hold for $K$.
\end{corollary}

\begin{corollary}
Let $K$ be a countable $C_1$ field of characteristic $0$ and let $L$ be a finitely generated extension of $K$ of cohomological dimension $i$. If Kato and Kuzumaki's conjectures hold for algebraic extensions of $L$ with cohomological dimension $<i$, then Kato and Kuzumaki's conjectures also hold for $L$.
\end{corollary}

Now, in \cite{ILA}, we introduced several variants of Kato and Kuzumaki's $C_i^q$ properties, by replacing hypersurfaces of low degree by homogeneous spaces of linear algebraic groups. Let us recall one of them.

\begin{definition}
Let $q$ be a non-negative integer. We say that a field $K$ has the $C_{\HS}^q$ property if, for each finite extension $L$ of $K$ and for each homogeneous space $Z$ under a smooth linear connected algebraic group over $L$, one has $N_q(Z/L) = K_q^M(L)$.
\end{definition}

Using Theorem \ref{th transfer principle car 0}, we can recover one of the main results in \cite{ILA} with a whole new approach, which focuses on fields rather than groups:

\begin{corollary}\label{cor Serre I}
Let $q$ be a non-negative integer. Every characteristic $0$ field with cohomological dimension at most $q+1$ has the $C_{\HS}^q$-property.
\end{corollary}

Together with the converse statement, which was easily settled in \cite[Prop. 3.2]{ILA}, this result shows that property $C_{\HS}^q$ is a good replacement for Kato and Kuzumaki's $C_1^q$ property, in the sense that it characterizes perfect fields of cohomological dimension $\leq q+1$. In the particular case where $q=0$, it recovers the zero-cycle version of Serre's conjecture I (Theorem \ref{conj Serre I}), which was proved by Steinberg \cite{Steinberg}, as well as an extension to homogeneous spaces with nontrivial stabilizers due to Springer \cite[III.2.4, Thm.~3]{SerreCohGal}. Corollary \ref{cor Serre I} is therefore in some sense a generalization of Serre's conjecture I and of Springer's theorem to higher-dimensional fields.

Thanks to the transfer principles, we now view this corollary under a new light, as a direct consequence of the case $q=0$. Indeed, Steinberg's and Springer's Theorems imply that fields of cohomological dimension $\leq 1$ have the $C_{\HS}^0$ property. Then Theorem \ref{th transfer principle car 0} tells us immediately that fields of cohomological dimension $\leq q+1$ have the $C_{\HS}^q$ property.

\subsection{Higher Serre's conjecture II}

\subsubsection*{The $C_\mathrm{sc}^q$ property}
Given the last application from the previous section, it is natural to ask whether one can find good replacements for the $C_2^q$ property that would characterize fields with cohomological dimension $\leq q+2$, while recovering a version of Serre's conjecture II (Conjecture \ref{conj Serre II}) for higher-dimensional fields. This is the purpose of the following definition.

\begin{definition}
Let $q$ be a non-negative integer. We say that a field $K$ has the $C^q_{\mathrm{sc}}$ property if, for any finite extension $L/K$ and any principal homogeneous space $Z$ under a semisimple simply connected $L$-group $G$, we have $N_q(Z/L)=\K_q(L)$. 

Similarly, given a type $\Lambda$ in the classification of semisimple absolutely almost simple simply connected groups (e.g. $\Lambda=A_n$, or $\Lambda=E_6$, or $\Lambda={^1}\!A_n$), we say that a field $K$ has the $C^q_{\Lambda}$ property if, for any finite extension $L/K$ and any principal homogeneous space $Z$ under a simply connected isotypical $L$-group $G$ of type $\Lambda$, we have $N_q(Z/L)=\K_q(L)$.
\end{definition}

Recall that a simply connected $L$-group $G$ is said to be isotypical of type $\Lambda$ if it is isomorphic to a finite product of Weil restrictions of the form $R_{M/L}(H)$ with $H$ an absolutely almost simple simply connected $M$-group $H$ of type $\Lambda$ and $M/L$ a finite separable extension. Properties $C^q_{\mathrm{sc}}$ and $C^q_{\Lambda}$ are then easily related to each other in the following way:

\begin{proposition}\label{prop prod}
Let $q$ be a non-negative integer. Assume that a field $K$ has the $C^q_{\Lambda}$ property for every type $\Lambda$. Then $K$ has the $C^q_{\mathrm{sc}}$ property.
\end{proposition}

\begin{proof}
This follows from Lemma \ref{lem produit} and the fact that any simply connected $K$-group $G$ is a product of Weil restrictions of semisimple almost simple simply connected groups.
\end{proof}

We are now ready to suggest the following higher variant of Serre's conjecture II:

\begin{conj}[Higher Serre's conjecture II]\label{higher Serre conj II text}
Let $q$ be a non-negative integer and $K$ a field. If $K$ has cohomological dimension at most $q+2$, then the field $K$ has the $C^{q}_{\mathrm{sc}}$ property.
\end{conj}

Before studying this conjecture in detail, we prove the following proposition, that shows that the converse holds for perfect fields (and that a partial converse holds for imperfect fields). In particular, the $C^{q}_{\mathrm{sc}}$ property should characterize perfect fields of cohomological dimension $\leq q+2$ and hence should be a good replacement for the $C_2^q$ property.

\begin{proposition}\label{1An} 
If a field $K$ has the $C^q_{{^1}\!A_n}$ property, then all of its finite extensions have separable cohomological dimension at most $q+2$. In particular, if $K$ is perfect, then its cohomological dimension is $\leq q+2$.
\end{proposition}

\begin{proof} Let $K$ be a field with the $C^q_{{^1}\!A_n}$ property. Fix a prime number $\ell$, and assume first that $\ell$ is different from the characteristic of $K$. Consider a finite extension $L$ of $K$ containing a primitive $\ell$-th root of unity and a symbol $\{a_1,\ldots,a_{q+3}\} \in H^{q+3}(L,\mu_{\ell}^{\otimes (q+3)})$. Let $A$ be the cyclic $L$-algebra $(a_1,a_2)_{\ell}$ and consider the $L$-variety $Z$ given by the equation $\mathrm{Nrd}_A(\mathbf{x}) = a_3$. It is a principal homogeneous space under the group $\SL_1(A)$. Since $K$ has the $C^q_{{^1}\!A_n}$ property, one can find finite extensions $L_1,\ldots,L_r$ of $L$ and elements $b_i\in \K_q(L_i)$ for $1\leq i\leq r$ such that:
$$\begin{cases} 
\forall\, i,\; Z(L_i)\neq \emptyset \\
\{a_4,\ldots,a_{q+3}\} = \sum_{i=1}^r N_{L_i/L}(b_i).
\end{cases} $$
By Theorem 24.4 of \cite{SuslinNorm}, the condition $Z(L_i)\neq \emptyset$ implies that the restriction of the symbol $\{a_1,a_2,a_3\} \in H^{3}(L,\mu_{\ell}^{\otimes 3})$ to $L_i$ is trivial. Hence:
$$\{a_1,\ldots,a_{q+3}\} = \sum_{i=1}^r N_{L_i/L}(\{a_1,a_2,a_3,b_i\}) = 0.$$
By the Bloch-Kato conjecture (\cite{Riou}), we deduce that the group $H^{q+3}(L,\mu_{\ell}^{\otimes (q+3)})$ is trivial. This being true for each finite extension $L$ of $K$ containing a primitive $\ell$-th root of unity, the field $K$ has $\ell$-cohomological dimension at most $q+2$ by Proposition \ref{prop Serre}.

Assume now that $\ell$ is equal to the characteristic of $K$. Consider a finite extension $L$ of $K$ and an element $x \frac{dy_1}{y_1} \wedge ... \wedge \frac{dy_{q+2}}{y_{q+2}} \in H^{q+3}_\ell(L)$. Let $A \in \mathrm{Br}(K)$ be the cyclic central simple algebra $ [x, y_1)$ and introduce the $\mathrm{SL}_1(A)$-torsor $Z$ given by the equation $\mathrm{Nrd}_A(\mathbf{x})=y_2$. Since $K$ has the $C^q_{{^1}\!A_n}$ property, one can find finite extensions $L_1,\dots,L_r$ of $L$ and elements $b_1 \in \K_q(L_1), \dots, b_r \in \K_q(L_r)$ such that:
$$\begin{cases} 
\forall i, \; Z(L_i)\neq \emptyset \\
\{y_3,\ldots,y_{q+2}\} = \sum_{i=1}^{r}  N_{L_i/L}(b_i) .
\end{cases} $$
According to Theorem 6 of \cite{gille}, the condition $Z(L_i)\neq \emptyset$ implies that: $$\mathrm{Res}_{L_i/L}\left( x \frac{dy_1}{y_1} \wedge \frac{dy_2}{y_2} \right)=0$$ for each $i$. By denoting by $\nu(q)_{L_i}$ the kernel of $\mathfrak{p}^q_{L_i}$ and by $\psi^q_{L_i}: \K_q(L_i)\rightarrow \nu(q)_{L_i}$ the differential symbol, and by using Lemma 9.5.7 of \cite{GS}, we get:
\begin{align*}
     x \frac{dy_1}{y_1} \wedge ... \wedge \frac{dy_{q+2}}{y_{q+2}}&=x \frac{dy_1}{y_1} \wedge\frac{dy_2}{y_2} \wedge  \left( \sum_i \mathrm{Tr}_{L_i/L}(\psi^q_{L_i}(b_i)) \right) \\&=\sum_i \mathrm{Tr}_{L_i/L}\left( \mathrm{Res}_{L_i/L}\left(x \frac{dy_1}{y_1} \wedge\frac{dy_2}{y_2}\right) \wedge  \psi^q_{L_i}(b_i)\right) = 0.
\end{align*}
We deduce that $H^{q+3}_\ell(L)=0$.
\end{proof}

\begin{remarque}
Proposition \ref{1An} shows that, if a field $K$ has the $C^q_{2}$ property, then all of its finite extensions have separable cohomological dimension at most $q+2$. Indeed, if $A$ is a cyclic algebra of index $d$, the equation $\mathrm{Nrd}_A(\mathbf{x}) = a$ has degree $d$ and $d^2$ variables, hence it naturally defines a degree $d$ hypersurface in $\bb P_K^{d^2}$.
\end{remarque}

\subsubsection*{Application of transfer principles to Higher Serre's conjecture II}

In order to study Conjecture \ref{higher Serre conj II text}, we may apply Theorem \ref{th transfer principle car 0} to torsors under semisimple simply connected groups and get the following conditional result:

\begin{theorem}\label{thm conditionnel}
If every countable field of characteristic $0$ and cohomological dimension $\leq 2$ has the $\Csc^0$ property (resp.~the $C_{\Lambda}^0$ property for $\Lambda$ a type in the classification of semisimple absolutely almost simple simply connected groups), then every field of cohomological dimension $\leq q+2$ has the $\Csc^q$ property (resp.~$C_\Lambda^q$ property), for any $q\geq 0$.
\end{theorem}

\begin{proof}
Assume first that the fields we are considering are of characteristic 0. Then we can apply directly Theorem \ref{th transfer principle car 0} with $n=q+2$, $m=2$, and $Z$ an arbitrary torsor under a semisimple simply connected group. The result follows in this case.\\

We are left then with the case of positive characteristic, which will follow by an \textit{ad hoc} transfer principle that will reduce this case to the one of characteristic 0. By Proposition \ref{prop prod}, we may reduce our study to a fixed type $\Lambda$. And since finite extensions of a field have the same cohomological dimension, it will suffice to prove that, for any field $K$ of cohomological dimension $q+2$ and characteristic $p>0$, any type $\Lambda$ in the classification of semisimple absolutely almost simple simply connected groups and any principal homogeneous space $Z$ under a simply connected isotypical $K$-group $G$ of type $\Lambda$, we have $N_q(Z/K)=\K_q(K)$. 

We henceforth fix such a field $K$ and type $\Lambda$. Let $G$ be a semisimple simply connected isotypical $K$-group of type $\Lambda$. Using Lemma \ref{lem produit}, we may assume that $G=R_{L/K}(H)$ for some finite separable extension $L/K$ and an absolutely almost simple simply connected $L$-group $H$ of type $\Lambda$. Let $H_0$ be the Chevalley group over $\mathbb{Z}$ such that $H$ is a twisted form of $H_0$ over $L$. According to Proposition 5 of \cite[IX.2, Prop.~5]{Bourbaki}, there exists a complete discrete valuation ring $A$ that has $p$ as a uniformizer, whose fraction field $\tilde{K}$ has characteristic $0$ and whose residue field is $K$. Let $\tilde{L}$ be the unramified extension of $\tilde{K}$ with residue field $L$ and let $B$ be its valuation ring. By \cite[XXIV, Thm.~1.3]{SGA3}, the group scheme $\mathrm{Aut}(H_0)$ is smooth over $\mathbb{Z}$, and hence, by \cite[XXIV, Prop.~8.1]{SGA3}, the map $H^1(B,\mathrm{Aut}(H_0))\rightarrow H^1(L,\mathrm{Aut}(H_0))$ is surjective. We deduce that there exists a semisimple simply connected absolutely almost simple group $\mathcal{H}$ of type $\Lambda$ over $B$ whose special fiber is $H$. We denote by $\tilde{H}$ its generic fiber and we set $\mathcal{G}:=R_{B/A}(\mathcal{H})$ and $\tilde{G}:=R_{\tilde{L}/\tilde{K}}(\tilde{H})$.

Now let $Z$ be a torsor under $G$. By \cite[Thm.~11.7, Rem.~11.8]{Brauer3}, the map $H^1(A,\mathcal{G})\rightarrow H^1(K,G)$ is bijective, and hence there exists a $\mathcal{G}$-torsor $\mathcal{Z}$ lifting $Z$. We denote by $\tilde{Z}$ the generic fiber of $\mathcal{Z}$. Then \cite[Cor.~to Thm.~3]{Kato} shows that $\tilde{K}$ has cohomological dimension $q+3$. Since $\tilde K$ has characteristic 0, it has the $C_\Lambda^{q+1}$ property, hence we have:
$$N_{q+1}(\tilde{Z}/\tilde{K})=\K_{q+1}(\tilde{K}).$$ 
Thus, given a symbol $\{u_1,...,u_{q}\}$ in $\K_{q}(K)$ and some liftings $\tilde{u}_1,...,\tilde{u}_{q}\in A$ of $u_1,...,u_{q}$, we can find finite extensions $\tilde K_1,\ldots,\tilde K_r$ of $\tilde K$ and elements $x_i\in \K_{q+1}(\tilde K_i)$ for each $i$ such that:
$$\begin{cases} 
\forall\, i,\; \tilde{Z}(\tilde K_i)\neq \emptyset, \\
\{p,\tilde{u}_1,...,\tilde{u}_{q}\} = \sum_{i=1}^rN_{\tilde K_i/\tilde{K}}(x_i).
 \end{cases} $$  
Denote by $\cal{O}_{i},K_i$ the corresponding valuation rings and residue fields. By using the compatibility of the norm morphism in Milnor $\mathrm{K}$-theory with the residue maps (cf.~\cite[Prop.~7.4.1]{GS}), we deduce that $\{u_1,...,u_{q}\}$ is a sum of norms coming from the $K_i$'s. Moreover, for each $i\in \{1,...,r\}$, the Grothendieck-Serre conjecture for discrete valuation rings (\cite{Nisnevich}) implies that the restriction morphism $H^1(\mathcal{O}_{i},\mathcal{G}) \rightarrow H^1(\tilde{K}_i,\tilde{G})$ is injective, so that $\cal Z(\cal O_i)\neq\emptyset$. We deduce that $Z(K_i)\neq \emptyset$, so that $\{u_1,...,u_{q}\}\in N_{q}(Z/K)$. This being true for any symbol in $\K_{q}(K)$, we get: $$N_{q}(Z/K)=\K_{q}(K).$$ 
This proves that $K$ has the $C_\Lambda^q$ property, as wished.
\end{proof}

As a consequence, we get the following unconditional result:

\begin{corollary}\label{cor incond}
Let $K$ be a field of cohomological dimension $\leq q+2$ with $q \geq 0$. Then $K$ satisfies properties $C^{q}_{A_n}$, $C^{q}_{B_n}$, $C^{q}_{C_n}$, $C^{q}_{D_n}$, $C^{q}_{G_4}$, $C^{q}_{F_2}$ and $C^{q}_{E_8}$. 
\end{corollary}

In particular, the only remaining cases left concerning Conjecture \ref{higher Serre conj II text} are those about groups of types $E_6$ and $E_7$.

\begin{proof}
The result follows from Theorem \ref{thm conditionnel} and \cite{BP} for every case with the exception of $D_4$ in the trialitarian case and $E_8$. For $E_8$, we use \cite[Thm.~8.4.1]{GilleLNM} and Lemma \ref{lem p Sylow} for $q=0$ instead of \cite{BP}. In the case of trialitarian $D_4$, we use the following result, which is well-known to experts, but which seems not to be published anywhere.
\end{proof}

\begin{proposition}
Let $K$ be a field of characteristic $0$ and cohomological dimension $\leq 2$. Let $G$ be an absolutely almost simple simply connected group of trialitarian type $D_4$ and let $Z$ be a $G$-torsor. Then $Z$ admits a zero-cycle of degree 1.
\end{proposition}

\begin{proof}
Applying Lemma \ref{lem p Sylow} with $q=0$, we immediately reduce to the case where either $K$ is $3$-special or $G$ is not trialitarian anymore. In the latter case, we are done by \cite{BP}. Thus, we may assume that $K$ is $3$-special and hence we are in the case of cyclic triality.

Let $L/K$ be the cyclic cubic extension such that $G_L$ is not trialitarian anymore. Then $G_L$ is an inner twist of the split group $G_0$ of type $D_4$, corresponding to a class $\alpha\in H^1(L,G_0^\ad)$. Now, since the center of $G_0$ is $2$-torsion and $L$ is $3$-special, we know that its cohomology is trivial and hence the map $H^1(L,G_0)\to H^1(L,G_0^\ad)$ is surjective. However, by \cite{BP} we know that $H^1(L,G_0)$ is trivial, and hence so is $\alpha$. This tells us that $G_L$ is split and then \cite[Prop.~3.2.11.(3)]{GilleLNM} shows that $G$ admits a maximal $K$-torus $T$ split by the cyclic cubic extension $L$.

Finally, let $G_1$ be the quasi-split form of $G$. By the same argument from above, we deduce that $G$ can be seen as a twist of $G_1$ defined by a class $\beta\in H^1(K,G_1)$. Applying \cite[Thm.~2.4.1]{BGL-Steinberg-2.0} to the embedding $T\to G$, we deduce that there exists an embedding $T\to G_1$ such that $\beta$ comes from $H^1(K,T)$. However, by \cite[Cor.~5.5.2.(1)]{GilleLNM}, we know that the map $H^1(K,T)\to H^1(K,G_1)$ is trivial, implying the triviality of $\beta$. We conclude then that $G$ is quasi-split and thus the result follows once again from \cite{BP}.
\end{proof}

\subsection{Classical Serre's conjecture II in positive characteristic}

To finish the article, we apply our transfer principles to prove that Serre's conjecture II in characteristic $0$ (over countable fields) implies Serre's conjecture II in positive characteristic:

\begin{theorem}\label{thm Serre II clasico}
Let $\Lambda$ be a type in the classification of semisimple absolutely almost simple simply connected groups. If Serre's conjecture II (Conjecture \ref{conj Serre II}) holds for torsors under semisimple simply connected groups of type  $\Lambda$ over characteristic 0 countable fields with cohomological dimension 2, then it holds for torsors under semisimple simply connected groups of type  $\Lambda$ over arbitrary fields.
\end{theorem}

\begin{proof}
Let $K$ be a field with cohomological dimension $\leq 2$, and let $Z$ be a torsor under a semisimple simply connected group $G$ of type  $\Lambda$ over $K$. Since both $G$ and $Z$ are of finite type, the subfield $K_0$ of $K$ generated by the coefficients of the (finitely many) equations defining the group $G$, its group law, the torsor $Z$ and the action of $G$ on $Z$, is finitely generated, hence countable. Then there exist $K_0$-forms $G_0$ and $Z_0$, of $G$ and $Z$ respectively, defined by the same equations. According to Proposition \ref{prop unc}, we can find a countable extension $L$ of $K_0$ contained in $K$ and of cohomological dimension $\leq 2$. If $K$ has characteristic $0$, then, by assumption, $Z_{0,L}$ is trivial, and hence so is $Z$.

We henceforth assume that $K$ has characteristic $p>0$. As in the proof of Theorem \ref{thm conditionnel}, one then can find:
\begin{itemize}
    \item[$\bullet$] a complete discrete valuation ring $B$ that has $p$ as a uniformizer, whose fraction field $\tilde{L}$ has characteristic $0$ and whose residue field is $L$;
    \item[$\bullet$] a semisimple simply connected group $\mathcal{G}_0$ of type  $\Lambda$ over $B$ with generic fiber $\tilde{G}_0$ and special fiber $G_{0,L}$;
    \item[$\bullet$] a torsor $\mathcal{Z}_0$ under $\mathcal{G}_0$ with generic fiber $\tilde{Z}_0$ and special fiber $Z_{0,L}$.
\end{itemize}
Theorem \ref{thm car p vers car 0} allows us to consider a totally ramified extension $\tilde L_\dagger$ of $\tilde L$ with cohomological dimension $\leq 2$ and integer ring $B_\dagger$. By the case where $K$ has characteristic $0$, that we have already solved, the class $[\tilde Z_{0,\tilde L_{\dagger}}]\in H^1(\tilde L_\dagger,\tilde G_0)$ is trivial. Hence there exists a finite subextension $\tilde L_1/\tilde L$ of $\tilde L_\dagger$, with integer ring $B_1$ such that $[\tilde Z_{0,\tilde L_1}]$ is trivial in $H^1(\tilde L_1,\tilde G_0)$. Now, by the Grothendieck-Serre conjecture (cf.~\cite{Nisnevich}), we know that the map $H^1(B_1,\cal G_{0,B_1})\to H^1(\tilde L_1,\tilde G_{0,\tilde L_1})$ is injective, and hence $[\cal Z_{0,B_1}]$ is trivial as well. Finally, since $\tilde L_1$ has residue field $L$, the specialization of $\cal Z_{0,B_1}$ at the closed point is $Z_{0,L}$, which is then trivial. The torsor $Z$ is therefore also trivial.
\end{proof}

Again, as a consequence of Theorem \ref{thm Serre II clasico}, we deduce the following unconditional result from \cite{BP}.

\begin{corollary}
Let $K$ be a field of cohomological dimension at most $2$. Then $H^1(K,G)=1$ for every semisimple simply connected $K$-group with no factors of types $E_6$, $E_7$, $E_8$ or trialitarian $D_4$.
\end{corollary}

This result was already proved in \cite{BFT} in a completely different way based on a case by case study following the classification of simply connected semisimple groups instead of focusing on fields.

\begin{appendix}

\section{Corrigendum}

In this appendix we fix a gap in the proof of Proposition \ref{proposition cle}. All the statements of the article remain unchanged.

\subsection*{The gap}
The proof of Proposition \ref{proposition cle} first deals with the particular case of the symbol \eqref{eqn a especifico}, to which the general case is then reduced. In this particular case, the proof of Fact~3 shows that all the degree $p$ subextensions of the $(\Z/p\Z)^{n+1}$-extension $\tL''/\tL$ are totally ramified, and concludes from this that $\tL''/\tL$ is totally ramified. This last implication is valid when the residue field $L$ is perfect, but it may fail when $L$ is imperfect, since the residue field extension may then be purely inseparable. We warmly thank K\c estutis \v Cesnavi\v cius for pointing this out to us.

\subsection*{What is discarded and what is kept}
The statement of Proposition \ref{proposition cle} is correct, so that all the other results
remain valid as they are. As for its proof:
\begin{itemize}
\item Lemma \ref{lem prod suma}, Lemma \ref{lem calculitos} and the proof of Proposition \ref{proposition cle} in the general case (that is,
the reduction to the case of the symbol \eqref{eqn a especifico}) are kept unchanged. Note that the proof of
Lemma \ref{lem calculitos} does not depend on this particular case.
\item The proof of Proposition \ref{proposition cle} for the symbol \eqref{eqn a especifico} must be replaced by the one given
below, which only relies on Lemma \ref{lem prod suma}, Lemma \ref{lem calculitos} and isomorphism \eqref{eq Kato iv}. In particular,
Lemma \ref{lem ram} and Facts~1, 2 and~3 are no longer needed.
\end{itemize}
The proof of Proposition \ref{proposition cle} obtained in this way uses neither the assumption that $K$ is
infinite nor the assumption that $\tL/\tK$ is Galois.

\subsection*{Proof of Proposition \ref{proposition cle} for the symbol \eqref{eqn a especifico}}
We keep the hypotheses and notation of Proposition \ref{proposition cle} and of \S3.2.1. Recall that symbols
of the form~\eqref{eqn a especifico} only appear when $\delta_0=\delta$: otherwise
$\Omega^\delta_L=0$ and there are no symbols of type~(b) in the proof of Proposition \ref{proposition cle}.
We assume this from now on, so that $\alpha_1,\dots,\alpha_\delta$ is a $p$-basis of both
$K$ and $L$. Every $x\in L$ can then be written in a unique way as
\[
x=\sum_{0\leq j_1,\dots,j_\delta<p}\lambda_{j_1,\dots,j_\delta}^p\,
\alpha_1^{j_1}\cdots\alpha_\delta^{j_\delta}
\qquad\text{with }\lambda_{j_1,\dots,j_\delta}\in L,
\]
and we set $\gamma(x):=\lambda_{0,\dots,0}$. The map $\gamma:L\to L$ is additive and
satisfies $\gamma(y^px)=y\,\gamma(x)$ for all $x,y\in L$.

Let $\beta\in K^\times$ and set $\pi':=\sqrt[p]{\pi\tilde\beta}$, where $\tilde\beta$ is a
lift of $\beta$ to $\cO_{\tK}$. Since $X^p-\pi\tilde\beta$
is an Eisenstein polynomial over $\tK$ and over $\tL$, the extensions $\tK(\pi')/\tK$ and
$\tL(\pi')/\tL$ are totally ramified of degree $p$ and have $\pi'$ as a uniformizer. Hence
the extension $\tL(\pi')/\tK(\pi')$ satisfies the hypotheses of Proposition \ref{proposition cle} with the
same residue field extension $L/K$, and we have $v_{\tL(\pi')}(p)=p\,v_{\tL}(p)$ and
$\epsilon_{\tL(\pi')}=p\,\epsilon_{\tL}$.

\begin{lemma}\label{lem corrigendum}
With the above notation, let $\kappa\in K^\times$ be the residue of
$-p\,(\pi\tilde\beta)^{-v_{\tL}(p)}$. Consider a symbol
\[
b=\{1+\pi^{\ell}\tilde d,\tas\}\in\kM_{\delta+1}(\tL),
\]
where $\ell>v_{\tL}(p)$ is an integer and $\tilde d\in\cO_{\tL}$ is a lift of some $d\in L$.
Then
\[
b|_{\tL(\pi')}=\{1+\pi'^{\,\ell+p\,v_{\tL}(p)}\,\tilde d',\tas\}
\]
for some lift $\tilde d'\in\cO_{\tL(\pi')}$ of $\kappa\,\gamma(\beta^{-\ell}d)$.
\end{lemma}

\begin{proof}
We argue by descending induction on $\ell$. If $\ell\geq\epsilon_{\tL}$, then
$\ell+p\,v_{\tL}(p)\geq\epsilon_{\tL(\pi')}$, and isomorphism \eqref{eq Kato iv} together with the fact
that $H^{\delta+1}_p(L)=0$ shows that both symbols are trivial for every choice of
$\tilde d'$. Assume now that $\ell<\epsilon_{\tL}$ and that the lemma holds for $\ell+1$.
Write
\[
\beta^{-\ell}d=\sum_{0\leq j_1,\dots,j_\delta<p}\lambda_{j_1,\dots,j_\delta}^p\,
\alpha_1^{j_1}\cdots\alpha_\delta^{j_\delta},
\]
so that $\lambda_{0,\dots,0}=\gamma(\beta^{-\ell}d)$, and let
$\tilde\lambda_{j_1,\dots,j_\delta}\in\cO_{\tL}$ be a lift of $\lambda_{j_1,\dots,j_\delta}$,
chosen to be $0$ if $\lambda_{j_1,\dots,j_\delta}=0$. Then
\[
\tilde\beta^{-\ell}\tilde d=\sum_{0\leq j_1,\dots,j_\delta<p}\mono+\pi\tilde c
\]
for some $\tilde c\in\cO_{\tL}$. Since $\pi^{\ell}=\pi'^{p\ell}\tilde\beta^{-\ell}$ and
$p\ell>v_{\tL(\pi')}(p)$, Lemma \ref{lem prod suma} applied to $\tL(\pi')$ gives
\[
b|_{\tL(\pi')}=\sum_{0\leq j_1,\dots,j_\delta<p}
\{1+\pi'^{p\ell}\mono,\tas\}
+\{1+\pi^{\ell+1}\tilde\beta^{\ell}\tilde c,\tas\}|_{\tL(\pi')}.
\]
By Lemma \ref{lem calculitos}(ii), applied to $\tL(\pi')/\tK(\pi')$, the terms of this sum with
$(j_1,\dots,j_\delta)\neq(0,\dots,0)$ are trivial. For the remaining one, the binomial
formula gives
\[
1+\pi'^{p\ell}\tilde\lambda_{0,\dots,0}^p
=(1+\pi'^{\ell}\tilde\lambda_{0,\dots,0})^p\,(1+p\,\pi'^{\ell}\tilde x)
\]
for some $\tilde x\in\cO_{\tL(\pi')}$ such that
$\tilde x\equiv-\tilde\lambda_{0,\dots,0}\equiv-\gamma(\beta^{-\ell}d)\pmod{\pi'}$. Since
$\pi'^{\,p\,v_{\tL}(p)}=(\pi\tilde\beta)^{v_{\tL}(p)}$, this term is equal to
$\{1+\pi'^{\,\ell+p\,v_{\tL}(p)}\,\tilde x',\tas\}$, where
$\tilde x':=p\,(\pi\tilde\beta)^{-v_{\tL}(p)}\tilde x$ is a lift of
$\kappa\,\gamma(\beta^{-\ell}d)$.

Finally, by the induction hypothesis, the last term is
equal to $\{1+\pi'^{\,\ell+1+p\,v_{\tL}(p)}\,\tilde c',\tas\}$ for some
$\tilde c'\in\cO_{\tL(\pi')}$. Lemma \ref{lem prod suma} then gives the result with
$\tilde d':=\tilde x'+\pi'\tilde c'$.
\end{proof}

\begin{proof}[Proof of Proposition \ref{proposition cle} for the symbol \eqref{eqn a especifico}]
If $\delta=0$, then $K$ is algebraically closed, since it has cohomological dimension $0$.
Thus $\tL=\tK$, every finite extension of $\tK$ is totally ramified, and it suffices to
take $\tK':=\tK(\sqrt[p]{1+\pi^{\ell_0}\tilde d_0})$. We assume henceforth that
$\delta\geq1$.

We prove that the conclusion of Proposition \ref{proposition cle} holds for every symbol of
the form
\begin{equation}\tag{$\ast$}
a=\{1+\pi^{\ell}\tilde d,\tas\}\in\kM_{\delta+1}(\tL),
\end{equation}
where $\pi$ is a uniformizer of $\tK$, $\tilde d\in\cO_{\tL}$ and $\ell>v_{\tL}(p)$ is an
integer. We proceed by induction on $\epsilon_{\tL}-\ell$, considering
simultaneously all the extensions $\tL/\tK$ that satisfy the hypotheses of
Proposition \ref{proposition cle} and have residue field extension $L/K$.

If $\ell\geq\epsilon_{\tL}$, then $a=0$ by isomorphism \eqref{eq Kato iv} and the fact that
$H^{\delta+1}_p(L)=0$. Assume then that $\ell<\epsilon_{\tL}$.

If $p\mid\ell$, then $a$ can be written in the form $(\ast)$ with $\ell$ replaced by
$\ell+1$, and we conclude by induction. Indeed, this is clear if
$\tilde d\in\pi\cO_{\tL}$. Otherwise, Lemma \ref{lem calculitos}(i) allows us to write $a$ as a sum of symbols
of the form
$\{1+\pi^{i}\tilde\lambda^p\tilde\alpha_1^{j_1}\cdots\tilde\alpha_\delta^{j_\delta},\tas\}$
with $\ell\leq i<\epsilon_{\tL}$, and by Lemma \ref{lem calculitos}(ii) those with $i=\ell$ are either
trivial or equal to a symbol $\{1+\pi^{i'}\tilde c,\tas\}$ with
$i'\geq v_{\tL}(p)+\ell/p>\ell$. It then suffices to put all these symbols together thanks
to Lemma \ref{lem prod suma}.

Assume finally that $(\ell,p)=1$. Fix an integer $m\geq1$ such that $p^{m\delta}>[L:K]$
and an element $\beta\in K^\times$ that will be chosen later. Set $\pi_0:=\pi$,
$\pi_1:=\sqrt[p]{\pi\tilde\beta}$ and $\pi_{i+1}:=\sqrt[p]{\pi_i}$ for $1\leq i<m$. We apply Lemma \ref{lem corrigendum} successively to the extensions
$\tL(\pi_i)/\tK(\pi_i)$ for $0\leq i<m$ and their uniformizers $\pi_i$, with the element
$\beta$ if $i=0$ and with the element $1$ otherwise. We get that $\tK(\pi_m)/\tK$ is
totally ramified of degree $p^m$, that $\tL(\pi_m)/\tK(\pi_m)$ satisfies the hypotheses of
Proposition \ref{proposition cle} with residue field extension $L/K$, and that
\[
a|_{\tL(\pi_m)}=\{1+\pi_m^{\ell_m}\tilde d_m,\tas\},\text{ where }\ell_m:=\ell+v_{\tL}(p)\,(p+p^2+\dots+p^m)=\ell+v_{\tL}(p)\frac{p^{m+1}-p}{p-1},
\]
for some $\tilde d_m\in\cO_{\tL(\pi_m)}$. Moreover, since $v_{\tL(\pi_i)}(p)=p^i\,v_{\tL}(p)$
and $\pi_i^{p^i}=\pi\tilde\beta$ for $1\leq i\leq m$, the constant $\kappa$ of Lemma \ref{lem corrigendum} is
the same at each of these $m$ steps, namely the residue of
$-p\,(\pi\tilde\beta)^{-v_{\tL}(p)}$. In other words $\kappa=\kappa_0\,\beta^{-v_{\tL}(p)}$,
where $\kappa_0\in K^\times$ is the residue of $-p\,\pi^{-v_{\tL}(p)}$, which does not depend
on $\beta$, and the residue $d_m$ of $\tilde d_m$ is given by
\[
d_m=\kappa\,\gamma\bigl(\kappa\,\gamma\bigl(\cdots\kappa\,\gamma(\beta^{-\ell}d)\cdots\bigr)\bigr)
=\gamma^m\bigl(\kappa^{\,p+p^2+\dots+p^m}\beta^{-\ell}d\bigr)
=\gamma^m\bigl(\kappa_0^{\,p+p^2+\dots+p^m}\beta^{-\ell_m}d\bigr).
\]

Now, the map $\phi:K\to L$, $z\mapsto\gamma^m(zd)$, is additive and satisfies
$\phi(x^{p^m}z)=x\,\phi(z)$ for all $x,z\in K$. In other words, $\phi$ is $K$-linear when
$x\in K$ acts on its source by multiplication by $x^{p^m}$. Since
$[K:K^{p^m}]=p^{m\delta}>[L:K]$, the map $\phi$ is not injective, so that there exists
$z\in K^\times$ such that $\gamma^m(zd)=0$. Since $\ell_m$ is prime to $p$, raising to the
power $\ell_m$ is a bijection of $K^\times/(K^\times)^{p^m}$, and hence we may choose
$\beta\in K^\times$ such that
$\kappa_0^{\,p+p^2+\dots+p^m}\beta^{-\ell_m}\in z\,(K^\times)^{p^m}$.
With this choice we get $d_m=0$, that is, $\tilde d_m\in\pi_m\cO_{\tL(\pi_m)}$. Thus
$a|_{\tL(\pi_m)}$ is a symbol of the form $(\ast)$ for the extension
$\tL(\pi_m)/\tK(\pi_m)$, the uniformizer $\pi_m$ and the integer $\ell_m+1$. Since
\[
\epsilon_{\tL(\pi_m)}-\ell_m=p^m\epsilon_{\tL}-\ell_m=\epsilon_{\tL}-\ell,
\]
the induction hypothesis provides a finite and totally ramified extension
$\tK'/\tK(\pi_m)$ of $p$-primary degree such that $a$ is trivial when restricted to
$\tK'\tL$. The extension $\tK'/\tK$ is then totally ramified of $p$-primary degree, which
concludes the proof.
\end{proof}

\subsection*{AI statement}
We have used Claude Fable 5.2 to write this appendix as quickly as possible once we were aware of the gap in the proof. The authors are still fully responsible for any other possible mistakes, both in the original article (which did not use any AI model) and in this appendix.

\end{appendix}

\end{document}